\documentclass[11pt]{amsart}

\usepackage{amsfonts,amsbsy, amsmath, amssymb, bbm, bm, color,  enumerate,tikz-cd,latexsym,amsopn,amstext, amsxtra,euscript,amscd,cite}

\usepackage{hyperref}

\usepackage[margin=2.6cm]{geometry}

\usepackage{helvet}

\RequirePackage{mathrsfs} \let\mathcal\mathscr

\numberwithin{equation}{section}

\newtheorem{theorem}{Theorem}
\numberwithin{theorem}{section}

\newtheorem{lemma}[theorem]{Lemma}
\newtheorem{proposition}[theorem]{Proposition}
\newtheorem{corollary}[theorem]{Corollary}
\newtheorem{conjecture}[theorem]{Conjecture}

\theoremstyle{definition}

\newtheorem*{remark}{Remark}
\newtheorem{definition}[theorem]{Definition}

\renewcommand{\phi}{\varphi}

\newcommand{\A}{\mathbf{A}}
\newcommand{\FF}{\mathbb{F}}
\newcommand{\ZZ}{\mathbb{Z}}
\newcommand{\Z}{\ZZ}

\newcommand{\NN}{\mathbb{N}}
\newcommand{\QQ}{\mathbb{Q}}
\newcommand{\Q}{\mathbb{Q}}
\newcommand{\RR}{\mathbb{R}}

\newcommand{\Gal}{{\rm Gal}}

\renewcommand{\emptyset}{\varnothing}

\renewcommand{\leq}{\leqslant}
\renewcommand{\geq}{\geqslant}

\newcommand{\z}{\mathbf{z}}
\renewcommand{\b}{\mathbf{b}}

\newcommand{\fa}{\mathfrak{a}}
\newcommand{\fb}{\mathfrak{b}}

\newcommand{\fp}{\mathfrak{p}}

\DeclareMathOperator{\Mod}{mod} 
\renewcommand{\bmod}[1]{\,(\Mod{#1})}

\newcommand{\sumstar}{\sideset{}{^*}\sum}
\newcommand{\sumprime}{\sideset{}{'}\sum}

\newcommand{\legendre}[2]{\genfrac{(}{)}{}{}{#1}{#2}}

\begin{document}
\title{Average Bateman--Horn for Kummer polynomials}

\author[Francesca Balestrieri and Nick Rome]{FRANCESCA BALESTRIERI\\ The American University of Paris\\ 5 Boulevard de la Tour-Maubourg\\ 75007 Paris\\ France \and\ NICK ROME\\ University of Michigan\\
Department of Mathematics\\
East Hall, 530 Church Street,
Ann Arbor, MI 48109}


\date{\today}

\maketitle

\thispagestyle{empty}

\begin{abstract}
For any prime $r \in \NN$ and almost all $k \in \mathbb{N}$ smaller than $x^r$, we show that the polynomial $f(n) = n^r + k$ takes the expected number of prime values, as $n$ ranges from 1 to $x$. As a consequence, we deduce statements concerning variants of the Hasse principle and of the integral Hasse principle for certain open varieties defined by equations of the form $N_{K/\mathbb{Q}}(\textbf{z}) = t^r + k \neq 0$, where $K/\Q$ is a quadratic extension. 
A key ingredient in our proof is a new large sieve inequality for Dirichlet characters of exact order $r$.
\end{abstract}

\section{Introduction}
One of the most central and classical problems in number theory is understanding how often polynomials take prime values. Indeed, infamous examples include the Prime Number Theorem (concerning the frequency with which the polynomial $x$ takes prime values), the Twin Prime Conjecture (concerning the frequency with which the polynomials $x$ and $x+2$ simultaneously take prime values) and Landau's problem (which asks how often the polynomial $x^2 +1$ takes prime values). A vast generalisation of these problems is provided by the following conjecture of Bateman and Horn~\cite{BH}.
\begin{conjecture}[Bateman-Horn]
Let $f_1, \ldots, f_r \in \ZZ[x_1, \ldots, x_n]$ be distinct irreducible polynomials of degree $d_1, \ldots, d_r$ respectively. If there is no prime which divides $\prod_{i=1}^r f_i(d)$ for every $d \in \NN$,  then 
\[
\sum_{d \leq X} \prod_{i=1}^r \Lambda(f_i(d)) \sim \frac{X}{d_1\cdots d_r}  \prod_p \frac{1 - n_p/p}{(1-1/p)^r},
\]
where $\Lambda$ is the von Mangoldt function and where $n_p$ is the number of solutions to $\prod_{i=1}^r f_i(x) \equiv 0 \bmod{p}$ in $\mathbb{Z}/ p\mathbb{Z}$.
\end{conjecture}
The only cases for which this conjecture is known to hold are for a single polynomial of degree 1 (this being Dirichlet's theorem on primes in arithmetic progressions), for polynomials given by the norm of a number field, or for polynomials in a large number of variables compared to their degree. The full strength of this conjecture is well out of reach at the moment. A more accessible statement is to average the conjecture over particular families of polynomials.  This is the content of our main theorem, which shows that an average version of the Bateman--Horn conjecture is true almost always for certain polynomials related to Dirichlet characters of exact order $r$, a prime.
\begin{theorem} \label{thm: average} Let $r \in \NN$ be a prime. Let $n_0 \in \Z$ and $M_0 \in \NN$.
For any fixed real $A, B  >0$, we have for $x^r(\log x)^{-A} \le y \le  x^r$ that
\begin{equation}\label{ABH}
\sum_{\substack{k \le y  }} \left | \sum_{\substack{n \le x \\ n\equiv n_0 \bmod{M_0} }} \Lambda(n^r+k) - \mathfrak{S}_{n_0, M_0}(k)x \right |^2 \ll \frac{yx^2}{(\log x)^B},
\end{equation}
where the singular series is given by 
$$
\mathfrak{S}_{n_0, M_0}(k) := \frac{\mathbbm{1}_{\gcd(M_0, n^r_0+k) = 1}}{\phi(M_0)} \prod_{\substack{ p \nmid 2 \cdot M_0\\ p \textup{ prime}}} \left (1 - \frac{n_p-1}{p-1} \right ).
$$
\end{theorem}

Theorem~\ref{thm: average} immediately implies the following result.

\begin{corollary} \label{cor}   Let $r \in \NN$ be a prime. Let $n_0 \in \Z$ and $M_0 \in \NN$.
For any fixed real $A, B, C >0$ and for $\mathfrak{S}_{n_0, M_0}(k)$ as defined in Theorem~\normalfont{\ref{thm: average}}, we have for $x^r(\log x)^{-A} \le y \le x^r$ that
$$
\sum_{\substack{n \le x \\ n \equiv n_0 \bmod{M_0} }} \Lambda(n^r +k) =  \mathfrak{S}_{n_0, M_0}(k) x  + O\left (\frac{x}{(\log x)^B} \right )
$$
for all $k$  up to $y$ with at most $O \left  (y(\log x)^{-C} \right  )$ exceptions.
\end{corollary}

When $M_0 = 1$, this subsumes previous results by Baier--Zhao~\cite{BZ},  Foo--Zhao~\cite{FZ}, and is analogous to Zhou\cite{zhou}. Recently, Sofos and Skorobogatov~\cite{SS} have independently investigated an average form of the Bateman--Horn conjecture and shown that, if we order polynomials by the size of their coefficients, then 100 \%  of them satisfy the Bateman--Horn conjecture. However, their work and their techniques cannot be used to provide insight into thin families of polynomials such as the ones that we consider in this paper (more discussion on this point will follow in the next subsection). The more arithmetic motivation for our work (as well as that of ~\cite{SS}) is that many results about the qualitative behaviour of rational points on varieties with a fibration structure are known to hold under the so-called Schinzel's hypothesis, which is a special case of the Bateman--Horn conjecture. Indeed the first such example of this is the use of Dirichlet's theorem on primes in arithmetic progressions in the proof of the Hasse--Minkowski theorem. The idea of using Schinzel's hypothesis and fibration techniques to deduce arithmetic statements (such as the Hasse principle or its refinements using the Brauer-Manin set, see e.g.  ~\cite{S} for definitions) is originally due to Colliot-Th{\'e}l{\`e}ne and Sansuc \cite{CTSschinzel} and was later futher developed by several authors in \cite{CTSD}, \cite{CTSkSD}, \cite{CTSkSD2}, and, more recently, in \cite{W}, \cite{Mi}, \cite{CWX}. In particular, it is known that, conjecturally on Schinzel's hypothesis,  the Brauer-Manin obstruction is the only one for the Hasse principle  and for the integral Hasse principle  on certain normic varieties (see e.g. \cite{CTS1}, \cite{CTSschinzel},  \cite{CTS2},  \cite{Mi}, \cite{CWX}). Furthermore, in the advent of breakthroughs in additive combinatorics by Green--Tao--Ziegler~\cite{GTZ}, it was possible to prove unconditional results of this type for fibrations whose degenerate fibres are all defined over $\QQ$ (see \cite{BMS} and \cite{HSW}).
Since the average Bateman--Horn result in Theorem~\ref{thm: average} also acts as a replacement on average for Schinzel's hypothesis, one can try to deduce, unconditionally, applications to the arithmetic statistics of the Hasse principle and integral Hasse principle. As a proof of concept, in this paper we give one such application. 
Consider the open affine variety
\begin{equation}\label{yztk} \mathcal{X}_{a,r,k} : y^2- az^2 =  t^r + k \neq 0 \end{equation}
for $a \in \Z-\{0,1\}$ squarefree, $r $ a prime, and $k$ a positive integer.

\begin{theorem}[Theorem \ref{chat}] Let $a \in \Z - \{0,1\}$ be squarefree and such that 2 does not ramify in $\mathcal{O}_{\Q(\sqrt{a})}$.  Let $r \in \NN$ be any prime such that $p \not\equiv 1 \bmod{r}$ for all primes $p|a$. Then for 100\%  of $k \in \NN$ (ordered naively by size) we have $\mathcal{X}_{a,r,k}(\Q) \neq \emptyset$. If, moreover, $\mathcal{O}_{\Q(\sqrt{a})}$ has narrow class number at most 2, then  for 100\%  of $k \in \NN$ (ordered naively by size) we have $\mathcal{X}_{a,r,k}(\Z) \neq \emptyset$.
\end{theorem}

The proof of this theorem consists in establishing, under the relevant conditions on $a$, that some variants of the Hasse principle and of the integral Hasse principle hold for $\mathcal{X}_{a,r,k}$. As our main tool, we use the analytic input together with a modification,  first appearing in \cite{SS}, of standard fibration method arguments. 
However, whilst averaging over large families in \cite{SS} allows them to get similar statements as the ones in this paper for an unspecified positive proportion of the associated varieties, by working closely with our much thinner family we are able to prove results for 100\% of the $\mathcal{X}_{a,r,k}$.
In other words, averaging over only one coefficient, as is the case in our thin family, makes the analytic part more difficult but yields more precision in the arithmetic applications.
This tight interplay between geometry and analysis is in the spirit of the results of \cite{BMS} and \cite{HSW}.
Theorem \ref{chat} immediately implies the following corollary, analogous to \cite{DH}.
\begin{corollary} Let 
\[\mathcal{L}:= \{  3, 7,  11, 15, 19, 35, 43, 51,  67, 91, 115, 123, 163, 187, 235, 267, 403, 427 \}.\]
Let $d \in \mathcal{L}$ and let $r_d \geq 3$ be any prime such that $p \not\equiv 1 \bmod{r_d}$ for all primes $p|d$. Then 100\% of positive integers $k$  can be written as 
\[ k = n_1^2 + d n_2^2 + n_3^{r_d}\]
for some $n_1, n_2, n_3 \in \Z$.

\end{corollary}
\begin{proof} Since for imaginary quadratic number fields  the class number $\leq 2$ problem has been solved and the class number and the narrow class number coincide, it is easy to check that the list $\mathcal{L}$ consists precisely of those squarefree $d \in \NN$ such that $\Q(\sqrt{-d})$ has (narrow) class number at most 2 and 2 does not ramify in $\mathcal{O}_{\Q(\sqrt{-d})}$. Hence, we can apply Theorem \ref{chat}.
\end{proof}


\subsection{Proof Outline}
Our proof of Theorem~\ref{thm: average} follows closely the outline in \cite{BZ}. The idea is to use the circle method, in a manner similar to Vinogradov's theorem, writing
\[
\sum_{n \leq x} \Lambda(n^r + k)
=
\int_0^1 \sum_{m \leq x^r + k}\Lambda(m)e(\alpha m) \sum_{n \leq x}e(- \alpha(n^r + k)) \mathrm{d}\alpha,\]
where we are using the standard notation $e(w) := \exp(2 \pi i w)$.
The difficulty with applying the Vinogradov approach for non-linear polynomials  is that, in order to detect primes, one must take very large major arcs, whereas, in order to detect integers represented by a polynomials of degree $r$, one typically has major arcs of length $x^{-r}$ (however our major arcs will in fact be of size $x^{1-r}$, c.f.\ Section \ref{section: major arc}). This means that our major arc contribution will not converge: indeed, the  bulk of our work consists in bounding the second moment of the tail of this contribution.
In the major arcs, the way we detect when an integer is an  $r$-th power is  by developing a large sieve for characters of exact order $r$, a result which we believe will be of independent interest. 
\begin{theorem} \label{largesieve} Let $r \geq 2$ be any integer. 
Let $(a_m)_{m \in \mathbb{N}}$ a sequence of complex numbers, supported on the squarefree integers. Then
\begin{equation}\label{eq:sieve}
\sum_{Q < q \leq 2Q} \sum_{\substack{\chi \Mod q\\ \chi^r = \chi_0\\ \chi \neq \chi_0}}
\left \vert  \sum_{M < m \leq 2M} a_m \chi(m) \right \vert^2
\ll_{\epsilon}
\Delta(Q,M)
\sum_{\substack{M < m \leq 2M\\ \gcd(m, r) =1}} \vert a_m \vert^2
\end{equation}
where $\chi$ denotes an order $r$ Dirichlet character modulo $q$, $\chi_0$ is the principal character, and 
\begin{equation}\label{eq:lsconst}
\Delta(Q,M) = (QM)^{\epsilon}\min\{Q^2 + M, Q^{3/2} + Q^{1/2}M, Q^{2/3}M+Q^{4/3}, Q +  M^{\frac{r+2}{3}}Q^{1/3} + M^{r - 1/2} \}.
\end{equation}
\end{theorem}
\begin{remark}
This result is likely not the best possible using our techniques. Further improvement will be the subject of future work.
\end{remark}
The cases $r=2,3,$ and 4 in Theorem \ref{largesieve} have been studied by \cite{HB2}, \cite{BY} and \cite{GZ}, respectively. 
In all cases, this is produced from a related large sieve result for the $r$-th power residue symbol.
We remark that in these previous works the field $\QQ(\zeta_r)$, over which this symbol is defined, was either $\QQ$ itself or an imaginary quadratic field of class number 1, thus greatly simplifying the analysis. In order to deal with the added complexities, we use the theory of sums over Hecke families as introduced in \cite{GL} and \cite{BGL}. The authors believe there is a great deal of untapped potential in applying this theory to similar character sum problems and we hope to inspire future research in this direction.

One could, however, replace our application of the large sieve by following more closely the work of Zhou \cite{zhou}. Specifically instead of using Theorem \ref{largesieve} to bound the quantity $\Psi_2$ in Section 6 of our paper, one could appeal to bounds on the Dedekind zeta function of the field $\QQ(\zeta_r)$. However since we believe that Theorem \ref{largesieve} could have numerous applications to similar problems on the average representation by Kummer polynomials of integers of arithmetic interest,  we have choosen to more closely work in the spirit of \cite{HB2}.
Moreover, the quadratic, cubic and quartic large sieve has already found extensive application to the study of the $L$-functions of such characters (e.g.\ \cite{GZ2} and \cite{GZ3}). We give below a quick example of one such application for our new sieve.
\begin{theorem} \label{application}
For $Q \geq 1$, we have
\[
\sum_{q \leq Q}\, \sumstar_{\substack{ \chi \bmod q \\ \chi^r = \chi_0}}
\vert L(1/2 + it, \chi) \vert^2 \ll_{\epsilon} Q^{\frac{7}{6}+\epsilon}(1 + \vert t \vert)^{{\frac{1}{2}+\epsilon}} +  Q^{\frac{4}{3}+\epsilon}.
\]
\end{theorem}

Finally, let us note how following the approach of \cite{SS} would not be tractable for the present problem. Using the dispersion method, we would have to open the square in the left-hand side of \eqref{ABH} and bound a sum of the form
\[
\sum_{k \leq y} \left[ \Lambda(n_1^r + k)\Lambda(n_2^r + k)
+
\Lambda(n_1^r + k)\mathfrak{S}_{n_0, M_0}(k)x
+
\Lambda(n_2^r + k)\mathfrak{S}_{n_0, M_0}(k)x
+
\mathfrak{S}_{n_0, m_0}(k)^2x^2\right].
\]
The most difficult term would then be the first: indeed, re-arranging slightly we could write the $k$-sum as
\[
\sum_{a \leq y+n_1^r} \Lambda(a) \Lambda(a+(n_2^r - n_1^r)).
\] 
Achieving an asymptotic formula for such a quantity is on the level of the Twin Prime Conjecture, and therefore unattainable with current techniques.

The paper is laid out as follows.
In Section \ref{sec:lemmas}, we gather several well-known lemmas that are necessary in the course of our proof. Section \ref{sec:sieve} is devoted to the proof of the large sieve result (cf. Theorem \ref{largesieve}). We prove our application of the large sieve, Theorem \ref{application}, in Section \ref{sec:interlude}. Theorem \ref{thm: average} is then proven in Sections \ref{section: major arc} - \ref{sec:together}. Finally, we make the application to rational and integral points in Section \ref{sec:chats}.

\subsection*{Acknowledgements} The authors are grateful to Tim Browning and Efthymios Sofos for useful conversations, and to Jean-Louis Colliot-Th{\'e}l{\`e}ne for his interest in our work. They are also indebted to the anonymous referees and Ben Green for the useful feedback and for pointing out some oversights in previous versions of this work. During part of this work, Francesca Balestrieri was supported by the European Union's Horizon 2020 research and innovation programme under the Marie Sklodowska-Curie grant 840684. During part of this work, Nick Rome was supported by EPSRC Studentship EP/N509619/1 179379.

\section{Preliminaries}\label{sec:lemmas}
For the reader's convenience, we collect here some results that we are going to use throughout the paper.
\begin{lemma}[P\'olya-Vinogradov {\cite[Theorem 12.5]{IK}}] \label{lem: polya-vinogradov}
Let $M,N,q \geq 1$.
For any non-principal character $\chi$ modulo $q$, we have
$$
\left | \sum_{M < n \le M+N} \chi(n) \right | \ll q^{1/2} \log q.
$$
\end{lemma}

\begin{lemma}[Number field large sieve{\cite[Theorem 1]{H}}] \label{lem: large sieve number field}
Let $K$ be a number field and $\mathfrak{r}$ denote an ideal in $K$. Suppose $u(\mathfrak{r})$ is a complex-valued function defined on the set of ideals in $K$. We have
$$
\sum_{\mathcal{N}(\mathfrak{f}) \le Q} \frac{\mathcal{N}(f)}{\Phi(\mathfrak{f})} \sumstar_{\chi  \Mod \mathfrak{f}} \left | \sum_{\mathcal{N}(\mathfrak{r}) \le z} u(\mathfrak{r}) \chi(\mathfrak{r}) \right |^2 \ll (z + Q^2) \sum_{\mathcal{N}(\mathfrak{r}) \le z} |u(\mathfrak{r})|^2,
$$
where $\mathcal{N}(\mathfrak{f})$ denotes the norm of the ideal $\mathfrak{f}$, $\Phi(\mathfrak{f})$ is Euler's totient function generalized to the setting of number fields, the $*$ over the summation over $\chi$ indicates that $\chi$ is a primitive character of narrow ideal class group modulo $\mathfrak{f}$ and the implicit constant depends on $K$.
\end{lemma}

\begin{lemma}[Duality principle {\cite[Theorem 228]{HLP}}]  \label{lem: duality principle}
For a finite square matrix $(t_{mn})$  with entries in the complex numbers, the following statements are equivalent:
\begin{enumerate}
\item  For any complex sequence $(a_n)$, we have
$$
\sum_{m} \left | \sum_{n} a_n t_{mn} \right |^2 \ll \sum_{n} |a_n|^2.
$$
\item For any complex sequence $(b_n)$, we have
$$
\sum_{n} \left | \sum_{m   } b_m t_{mn} \right |^2 \ll \sum_{m} |b_m|^2.
$$
\end{enumerate}
\end{lemma}

\begin{lemma}[Perron formula \cite{D}] \label{lem: perron}
Suppose that $y \neq 1$ is a positive real number. Then, for $c,T>0$, we have
$$
\frac{1}{2 \pi i} \int_{c-iT}^{c+iT} \frac{y^s}{s} \  \mathrm{d} s=
\begin{cases}
1 + O \left (y^{c} \min \{1,T^{-1} |\log y|^{-1} \} \right ) &\mbox{if $y>1$,} \\
O \left  (y^{c} \min \{1,T^{-1} |\log y|^{-1} \} \right )  & \mbox{otherwise.}
\end{cases}
$$
\end{lemma}

\begin{lemma}[Weyl bound {\cite[Proposition 8.2]{IK}}] \label{lem: weyl}
If $f(x)= \alpha x^d+ \ldots + a_0$ is a polynomial with real coefficients  and $d \ge 1$, then
$$
\left | \sum_{n \le N}e(f(n)) \right | \le 2 N \left  \{N^{-d} \sum_{-N < \ell_1, \ldots, \ell_{d-1} <N} \min \left ( N , \frac{1}{\left \lVert \alpha d! \prod_{i=1}^{d-1} \ell_i \right \rVert} \right) \right \}^{2^{1-d}}.
$$
Here $\lVert x \rVert$ is the distance of $x$ to the nearest integer .
\end{lemma}

\begin{lemma} [Mikawa ~\cite{M}]\label{lem: mikawa}
Let
$$
\mathfrak{J}(q,\Delta) = \sum_{\chi  \bmod q} \int_{N}^{2N} \left |  \sum_{t < n < t+q\Delta}^{\#} \Lambda(n) \chi(n) \right |^2 \mathrm{d} t
$$
where the $\#$ over the summation symbol means that if $\chi=\chi_0$, then $\chi(n) \Lambda(n)$ is replaced by $\Lambda(n)-1$. Let $\varepsilon,A,B>0$ be given. If $q \le (\log N)^B$ and $N^{1/5+\varepsilon} < \Delta < N^{1-\varepsilon}$, then we have
$$
\mathfrak{J}(q,\Delta) \ll_{\varepsilon,A,B} (q\Delta)^2 N(\log N)^{-A}.
$$
\end{lemma}

\begin{lemma}[{Gallagher~\cite[Lemma 1]{G}}] \label{lem: gallagher}
Let $2 < \Delta < N/2$ and $N < N' < 2N$. For arbitrary complex sequence $(a_n)_{n \in \mathbb{N}}$, we have
$$
\int_{|\beta|< \Delta^{-1}} \left | \sum_{N < n < N'} a_n e(\beta n) \right |^2 d\beta \ll \Delta^{-2} \int_{N-\Delta/2}^N \left | \sum_{\max \{t,N \} < n < \min \{t+\Delta/2,N' \}} a_n \right |^2 \ \mathrm{d} t. 
$$
\end{lemma}

\begin{lemma}[{Bessel \cite[Thm 1, $\S$47]{H}}] \label{lem: bessel}
Let $\phi_1, \phi_2 , \ldots, \phi_R$ be orthonormal members of an inner product space $V$ over $\mathbb{C}$ with inner product $(\cdot, \cdot)$ and let $\xi \in V$. Then
$$
\sum_{r=1}^{R} |(\xi, \phi_r)|^2 \le (\xi, \xi).
$$
\end{lemma}

\begin{lemma}[{\cite[Theorem 5.35]{IK}}] \label{lem: zero free region}
Let $K /\QQ$ be a number field, $\xi$ a Hecke Grossencharakter modulo $(\mathfrak{m}, \Omega)$ where $\mathfrak{m}$ is a non-zero integral ideal in $K$ and $\Omega$ is a set of real infinite places where $\xi$ is ramified. Let the conductor $\Delta = |d_K|N_{K/\QQ} \mathfrak{m}$. There exists an absolute effective constant $c' >0$ such that the L-function $L(\xi, s)$ of degree $d=[K: \QQ]$ has at most a simple real zero in the region
$$
\sigma > 1-\frac{c'}{d \log \Delta (|t| +3)}.
$$
The exceptional zero can occur only for a real character and it is strictly less than $1$.
\end{lemma}

\begin{lemma}[Inverse Mellin transform{\cite[Lemma 12]{HB2}}]\label{mellin}
Let $\rho:\RR \rightarrow \RR$ an infinitely differentiable function whose derivatives satisfy \[ \frac{\mathrm{d}^k}{\mathrm{d}x^k} \rho(x) \ll_A \vert x \vert^{-A}.\]
Let \[ \rho_+(s)  = \int_0^{\infty} \rho(x) x^{s-1} \mathrm{d}x\]
and
 \[ \rho_-(s)  = \int_0^{\infty} \rho(-x) x^{s-1} \mathrm{d}x.\]
Then $\rho_+(s)$ and $\rho_-(s)$ are holomorphic for $\Re(s)=\sigma > 0$ and satisfy $\rho_{\pm}(s)  \ll_{A,\sigma}  \vert s \vert^{-A}$. Moreover, for $\sigma > 0$, we have
\[
\rho(\pm x)  = \frac{1}{2 \pi i} \int_{\sigma-i\infty}^{\sigma+i\infty} \rho_{\pm}(s) x^{-s} \mathrm{d}s.
\]
\end{lemma}

\begin{lemma}[Hua's lemma with congruences]\label{hua}
For $\alpha \in (0,1)$, let $T(\alpha) = \sum \limits_{\substack{ n \leq x \\ n \equiv a \bmod{q}}} e(\alpha n^k)$. Then, for any $\epsilon>0$,
\[
\int_0^1 \vert T(\alpha) \vert^{2^k} \mathrm{d}\alpha \ll \left(\frac{x}{q}\right)^{2^k-k+\epsilon}.
\]
\end{lemma}
\begin{proof}
The proof is very similar to that of e.g. \cite[Lemma 3.2]{DB}.
\end{proof}

\section{The Large Sieve}\label{sec:sieve}
Let $r \geq 2$ be a fixed integer. Throughout this section let $k$ be a number field containing all the $r$-th roots of unity and $\mathcal{O}_k$ the ring of integers of $k$. 
Given an integral ideal $\mathfrak{c}$, we will denote by $I(\mathfrak{c})$ and $I^*(\mathfrak{c})$, the set of integral ideals and fractional ideals of $\mathcal{O}_k$, respectively, which are coprime to $\mathfrak{c}$. We will define $I(S)$ and $I^*(S)$ analogously, for $S$ a finite set of places of $k$.

 For $a \in k$, let $S_a$ be the set of places of $k$ which either divide $r$ or ramify in $k(a^{1/r})/k$. 
At each prime $\fp \in I^*(\mathfrak{c})$, we have the Frobenius automorphism $F_a(\fp)$. We extend this multiplicatively to all fractional ideals to get the Artin map $F_a: I^*(S_a) \rightarrow \Gal(k(a^{1/r})/k)$.
For any $\fp \in I^*(\mathfrak{c})$, we have
\[
F_a(\fp)(a^{1/r}) = \legendre{a}{\fp} a^{1/r},
\]
where $\legendre{a}{\fp}$ is some $r$-th root of unity. The symbol $\legendre{a}{\fp}$ is independent of the choice of $a^{1/r}$ in this construction and is defined to be the \emph{$r$-th power residue symbol}. Indeed, \[ \legendre{a}{\fp} = 1 \iff a \text{ is an $r$-th power in } k_{\fp}.\]
Note that we may extend this multiplicatively to a symbol $\chi_a(\fb) = \legendre{a}{\fb}$ for any $\fb \in I^*(S_a)$.

One of the troubles that we run into if we try to extend the techniques in the papers  \cite{HB2}, \cite{BY} and \cite{GZ} mentioned in the introduction is that, in those papers, the authors work with power residue symbols $\chi_a(\fb)$ taking one integer argument and one ideal argument. In order to prove our large sieve result for any $r \geq 5$, we require a generalisation of the power residue symbol that takes two ideals as arguments. To achieve this, we need to introduce the notion of a Hecke family of characters.
\begin{definition} Let $r \geq 2$ be an integer. Let $k$ be a number field containing the group $\mu_r$ of all the $r$-th roots of unity. Fix an ideal $\mathfrak{c}$ of $\mathcal{O}_k$.  An \emph{$r$-th order Hecke family (with respect to the ideal $\mathfrak{c}$)} is a collection \[  \{ \chi_{\fa} : \fa \in I(\mathfrak{c}), \mu^2(\fa) =1 \}\]
of primitive Hecke characters of trivial infinity type satisfying the following three properties:
\begin{enumerate}
\item The order of each character $\chi_{\fa}$ divides $r$.
\item There exists a finite group $G$, a homomorphism $[\cdot ]$ from $I(\mathfrak{c})$ to $G$, and a map $C: G \times G \rightarrow \mu_n$ such that \[ \chi_{\fa}(\fb) = \chi_{\fb}(\fa) C([\fa],[\fb])\] for all coprime ideals $\fa, \fb \in I(\mathfrak{c})$. Note that we can think of $C([\fa],[\fb])$ as a sort of reciprocity law factor.
\item For all coprime ideals $\fa, \fb \in I(\mathfrak{c})$ satisfying $[\fa]=[\fb]$, we have that $\chi_{\fa}\overline{\chi_{\fb}}$ is a primitive Hecke character modulo $\fa\fb$.
\end{enumerate}
\end{definition}

In \cite{FF}, Fisher and Friedberg construct, for each ideal $\fa \in I(S)$, a Hecke character $\chi_{\fa}$ that generalises the power residue symbol $\left(\frac{a}{\cdot}\right)$ in the sense that, for $a \equiv1 \Mod \mathfrak{c}$, we have $\chi_{(a)}=\chi_a$. One of the major ingredients in the proof of Theorem \ref{largesieve} is the following correspondence between Dirichlet characters of order exactly $r$ and $r$-th power residue symbols.

\begin{proposition}\label{correspondence}
There is a one-to-one correspondence between the set of all primitive Dirichlet characters of exact order $r$ and conductor $p$, where $p$ splits completely in the ring of integers of the $r$-th cyclotomic field $\Q(\zeta_r)$ and the set of all $r$-th power residue symbols $\chi_{\fp}$ at prime ideals $\fp$ in $\Z[\zeta_r]$ which lie above $p$.
Moreover, there are no primitive Dirichlet characters of exact order $r$ and conductor $p^\alpha$ for $\alpha \geq 2$.

By multiplicativity, one can extend the correspondence to squarefree conductors of the form $q = \prod_{i=1}^s p_i$, where the primes $p_i$ all split completely over $\Q(\zeta_r)$.
\end{proposition}

\begin{proof}
To classify all primitive Dirichlet characters of order $r$ and conductor $q$, we first note that it suffices, by multiplicativity, to consider $q = p^\alpha$ where $p$ is a rational prime and where $\alpha \geq 1$. We start with the case $\alpha=1$.
If there are primitive characters of modulus $p$ and order $r$, then it must be the case that $r \mid p-1$ and therefore that $p$ splits completely in $\Q(\zeta_r)$. Conversely, if $p$ splits completely in $\Q(\zeta_r)$ (and thus $p \equiv 1 \ (\textup{mod } r)$), say as $p \mathcal{O}_{\Q(\zeta_r)} = \fp_1 \cdots \fp_{\varphi(r)}$, then associated to each $\fp_i$ there is a map $m \mapsto \chi_{\fp_i}((m))$, which is a Dirichlet character of order $r$ and modulus $N(\fp_i) = p$. Furthermore, these maps represent all such characters by the following simple counting argument. For any Dirichlet character $\chi$ of modulus $p^{\alpha}$, we have (see ~\cite[Chapter 4]{D})
\[ \chi(n) = e^{2\pi i m \nu(n)/\varphi(p^{\alpha})},\]
where $m$ is some fixed integer and $\nu(n)$ is the index of $n$ relative to a particular primitive root of $p$; that is, if $g$ is a fixed primitive element modulo $p^\alpha$ (and hence, even for a general $\alpha$, a primitive element modulo $p$), then $\nu(n)$ is defined (up to adding multiples of the order of $g$ in $(\mathbb{Z}/p^{\alpha} \mathbb{Z})^\times$) by $g^{\nu(n)} \equiv n \ (\textup{mod } p^\alpha)$.  In our case, since $\alpha = 1$, we have $\varphi(p^\alpha) = p-1$.

Since $\chi$ has order $r$, we have, for all $n$, that
\[ 1 = \chi^r(n) = e^{2\pi i m r  \nu(n)/(p-1)}.\]
Hence, for all $n$, we need $m r  \nu(n)/(p-1) \in \mathbb{Z}$. Note that $\nu(g) = 1$ (and $\gcd(g,p)=1$), so in fact we need $m r /(p-1) \in \mathbb{Z}$, i.e. $m = ((p-1)/r)  w$, for some $w \in \mathbb{Z}$. Hence, 
\[ \chi(n) = e^{2\pi i m \nu(n)/p-1} =  e^{2\pi i w \nu(n)/r}.\]
Now, from the above it is easy to see that, if $w_1, w_2 \in \mathbb{Z}$ and $w_1 \equiv w_2  \ (\textup{mod } r)$, then $e^{2\pi i w_1 \nu(n)/r} = e^{2\pi i w_2 \nu(n)/r}$. Since all the characters of modulus $p$ and of order dividing $r$ are produced by varying $w$, we just need to consider $w \in \{0, 1, ..., r-1\}$;  hence, there can be at most $r$ Dirichlet characters of modulus $p$ and of order dividing $r$. Moreover, for $\chi$ to be of order exactly $r$, we need $w$ to be coprime to $r$. Hence, we get at most $\varphi(r)$ distinct characters of modulus $p$ and of order exactly $r$. Since the maps $m \mapsto \chi_{\fp_i}((m))$  already give $\phi(r)$ distinct Dirichlet characters of modulus $p$ and of order  $r$, it follows that these maps represent, in fact, all the (primitive) Dirichlet characters of modulus $p$ and of order $r$.

If $q = p^\alpha$ for $\alpha \geq 2$, then we claim that there can be no primitive characters of modulus $q$ and of exact order $r$, as they are always induced by characters of modulus $p$. Indeed, as we have seen, for a Dirichlet character $\chi$ of modulus $p^\alpha$ we have
\[ \chi(n) = e^{2\pi i m \nu(n)/\phi(p^\alpha)} =  e^{2\pi i m \nu(n)/(p^{\alpha-1}(p-1))},\]
where $m$ is some fixed integer and $\nu(n)$ is the index of $n$ relative to a particular primitive root $g$ of $p^\alpha$.
But since the order of $\chi$ is $r$ and $p-1 = r u$ for some integer $u$, we have, for all $n$, that
\[ 1 = \chi^r(n) = e^{2\pi i m   \nu(n)/(p^{\alpha-1} u)}.\]
So we need $m \nu(n)/(p^{\alpha-1}u) \in \mathbb{Z}$ for all $n$. Since  $\nu(g) = 1$, and $(g,p) = 1$,  it follows that $p^{\alpha-1} | m$, say $m = m' p^{\alpha-1} $. Hence,
\[  \chi(n) = e^{2\pi i m \nu(n)/\phi(p^\alpha)} =  e^{2\pi i m' \nu(n)/(p-1)}.\]
The right-hand side now can be seen as a Dirichlet character of modulus $p$ and order $r$: note that $\nu(n)$ is also the index of $n$  relative to the fixed primitive root $g$ when seen as a  primitive element of $p$ (i.e. $g^{\nu(n)} = n   \ (\textup{mod } p)$).  This shows that $\chi$ is induced by a (primitive) Dirichlet character of order $r$ and modulus $p$. 
\end{proof}

\begin{remark}
The correspondence between primitive Dirichlet characters of order $r$ modulo squarefree $q$ and products of $r^\text{th}$ power residue symbols breaks down for composite $r$. Instead the Dirichlet characters of order $r$ and modulus $q$ which correspond to $r^{\mathrm{th}}$ power residue symbols are those given by a product of characters of order $r$ for each prime $p$ dividing $q$. In the case that $r$ is prime, this is all Dirichlet characters mod $q$ of order $r$ however when $r$ is composite this criterion is more selective. Indeed suppose $p_1$ and $p_2$ are two distinct primes and let $\chi_1 \Mod p_1$ and $\chi_2 \Mod p_2$ be two Dirichlet characters of order 4 and 2, respectively. Their product $\chi_1\chi_2$ is a Dirichlet character of order 4 modulo $p_1p_2$ but it does not correspond to a quartic residue symbol. In order to extend the large sieve to composite $r$ a new idea is necessary.

Composite order large sieves have been claimed in \cite{GZ} and \cite{BY} however the authors believe the statements should be amended. The proofs as written only handle primitive characters $\chi$ of order 4, and 6 respectively, such that $\chi^2$, and $\chi^3$ in the latter, is also primitive.
\end{remark}

\subsection{Proof of Theorem \ref{largesieve}}
We start by introducing several norms associated to the sum which we aim to estimate. The comparison and estimation of these norms will ultimately yield our desired bound.
Firstly, let 
\[
B_1(Q,M) := \sup_{(a_m) \not \equiv 0} \vert \vert a_m \vert \vert^{-2} \sumprime_{N(\fa) \sim Q} \mu^2(\fa) 
\left\vert \sum_{m \sim M} \mu^2(m) a_m  \chi_{(m)}(\fa) \right\vert^2,
\]
where we have used the convention that $m \sim M$ means a sum over the range $M < m \leq 2M$, $\sumprime$ means a sum over all those ideals whose prime ideal divisors lie above completely split rational primes, and $\vert \vert a_m \vert \vert^2 := \sum_{m \sim M} \vert a_m \vert^2$. 
By the correspondence in Proposition \ref{correspondence}, every primitive Dirichlet character of modulus between $Q$ and $2Q$ of order $r$ is uniquely represented by a Hecke character $\chi_{(.)}(\fa)$ arising in the $\fa$ sum. Our bound \eqref{eq:sieve} will follow from a bound on $B_1(Q,M)$, since we can reduce any non-primitive character to the primitive character inducing it, and apply the bound to that.
We define the norm $B_2(Q,M)$ similarly, but loosening the restriction that primes dividing $\fa$ lie above completely split primes to just that $\fa$ must be coprime to $r \mathcal{O}_{\Q(\zeta_r)}$ and dropping the squarefreeness condition on $\fa$.
The bound
\[
B_1(Q,M) \leq B_2(Q,M) 
\]
is then trivial. Moreover, if $W$ is a smooth weight function with compact support in $(0, \infty)$ which is at least 1 on the interval $[1,2]$, then 
\[
B_2(Q,M) \leq \sup_{(a_m)} \vert \vert a_m \vert \vert^{-2} \sumstar_{m_1,m_2} a_{m_1} \overline{a_{m_2}} \sum_{\substack{\fa \subset \mathcal{O}_k\\(\fa, r \mathcal{O}_k)=1}} W\left(\frac{N(\fa)}{Q}\right) \chi_{(m_1)}(\fa) \overline{\chi_{(m_2)}(\fa)}.
\]
Introducing an additional coprimality condition, we define
\[
B_3(Q,M) 
:=
\sup_{(a_m)} \vert \vert a_m \vert \vert^{-2} \sumstar_{\gcd(m_1,m_2)=1} a_{m_1} \overline{a_{m_2}} \sum_{\substack{\fa \subset \mathcal{O}_k\\(\fa, r \mathcal{O}_k)=1}}  W\left(\frac{N(\fa)}{Q}\right) \chi_{(m_1)}(\fa) \overline{\chi_{(m_2)}(\fa)}.
\]
We define $C_1(M,Q)$ as the dual sum to $B_1$, namely
\[
C_1(M,Q) := 
\sup_{(b_{\fa}) \not \equiv 0} \vert \vert b_{\fa} \vert \vert^{-2} \sum_{m \sim M} \mu^2(m) 
\left \vert \sumprime_{N(\fa) \sim Q} \mu^2(\fa)b_{\fa} \chi_{\fa}((m)) \right \vert^2.
\]
Similarly, we define $C_2(M,Q)$ by removing from $C_1(M,Q)$ the restriction that $m$ be squarefree. 

A large sieve over Hecke families was established in \cite{GL} and \cite{BGL}. From the analysis in these papers,  we take the following bounds.
\begin{lemma}
For any $Q,M \geq 1$ and any $\epsilon>0$, we have
\begin{align}
B_1(Q_1,M) &\ll B_1(Q_2,M) \text{ for } Q_1, M \geq 1 \text{ and } Q_2 \geq Q_2 \geq CQ_1\log(2Q_1M) \label{eq:lem3}\\
B_2(Q,M) \ &\ll M^{\epsilon} B_3\left( \frac{Q}{\Delta_1}, \frac{M}{\Delta_2}\right) \text{ for some } 1\leq\Delta_1\ll\Delta_2 \label{eq:lem5}.
\end{align}
\end{lemma}

These bounds follow exactly the proofs of \cite[Lemma 3.1 and 3.2]{BGL} or \cite[Lemma 9 and Lemma 7]{HB2}.
The other bounds that we need are proved in the following lemma.

\begin{lemma}\label{lem:bounds}
For any $Q,M \geq 1$ and any $\epsilon>0$, we have
\begin{align}
C_2(M,Q) &\ll (QM)^{\epsilon}(M+ Q^{2}),\label{eq:lem1}\\
B_3(Q,M)\ &\ll Q + (QM)^{\epsilon}\frac{Q}{M^{\phi(r)}} \max \limits_{1 \leq K \leq (QM)^{\epsilon}M^{2\phi(r)} Q^{-1}} B_2(K,M) \nonumber\\
&\phantom{\ll Q }+\frac{M^{3\phi(r)}}{Q} \sum_{K >M^{2 \phi(r)}/Q} K^{-2}B_2(K,M),\label{eq:lem6}\\
B_2(Q,M) &\ll (\log 2Q)^3 Q^{1/2}X^{-1/2} B_1(XQ^{\epsilon},M) \text{ for some } X \text{ with } 1 \leq X \leq Q  \label{eq: lem4}\\
C_2(M,Q) &\ll M^{\epsilon}Q^{1-1/\nu} \sum_{j=0}^{\nu -1} C_2(2^jM^{\nu},Q)^{1/\nu} \label{eq:lem2}.
\end{align}
\end{lemma}

In order to establish these, it will be necessary to employ the Hecke family version of the Poisson summation formula \cite[Lemma 2.2]{BGL}. Write $\widehat{W}(s)$ for the Mellin transform of $W$ and $d_k$ for the discriminant of a number field $k$. Let
\[
\dot{W}(x) = \frac{1}{2\pi i} \int_{c-i\infty}^{c+i \infty} \widehat{W}(1-s) \left(\frac{(2\pi)^s \Gamma(s)}{\Gamma(1-s)}\right)^{d/2} \vert d_{\QQ(\zeta_r)} \vert^{-s/2} x^{-s} \mathrm{d}s,
\] for $c>0$.
\begin{lemma}[{\cite[Lemma 2.2]{BGL}}]\label{poisson}
Let $M>0$ and $\chi$ a primitive ray class character of conductor  $\mathfrak{f}$. Then we have
\[
\sum_{\fa} W\left(\frac{N(\fa)}{M}\right) \chi(\fa) = \frac{M\epsilon(\chi)}{\sqrt{N(\mathfrak{f})}} 
\sum_{\fa} \dot{W}\left(\frac{M N(\fa)}{N(\mathfrak{f})}\right) \overline{\chi}(\fa).
\]
\end{lemma}

\begin{proof}[Proof of Lemma \ref{lem:bounds}]
The first bound follows from a simple application of the usual large sieve in number fields (Lemma \ref{lem: large sieve number field}). 
Next we note that when $M<1$, the bound $B_3(Q,M) \ll Q$ is trivial.
We know that 
\[
B_3(Q,M)
\leq
\sumstar_{\gcd(m_1,m_2) =1 }
a_{m_1} \overline{a_{m_2}}
\sum_{\gcd(\fa, r \mathcal{O}_k)=1}
W \left(\frac{N(\fa)}{Q} \right) \chi_{(m_1)}(\fa) \overline{\chi_{(m_2)}(\fa)}.
\]
Using \cite[Lemma 2.3]{BGL} we can replace $\chi_{(m_1)} \overline{\chi_{(m_2)}}$ by the underlying primitive Hecke character and then apply the Hecke family version of Poisson summation \cite{BGL} to get
\[
B_3(Q,M)
\leq
Q
\sumstar_{\gcd(m_1,m_2) =1 }
\frac{a_{m_1} \overline{a_{m_2}}}{(m_1m_2)^{\phi(r)/2}}
\sum_{\gcd(\fa, r \mathcal{O}_k)=1}
\dot{W} \left(\frac{N(\fa)Q}{(m_1m_2)^{\phi(r)}} \right) \chi_{(m_1)}(\fa) \overline{\chi_{(m_2)}(\fa)}.
\]
We will split the $\fa$ sum into dyadic intervals and if the interval has length $K>M^{2\phi(r)}/Q$ then we can apply the fast decay of the Mellin transform of $W$ hence for any $A>0$ we have
\begin{align*}
B_3(Q,M) 
&\ll
\frac{Q}{M^{\phi(r)}} \sum_{\substack{K \leq M^{2\phi(r)}/Q\\ \text{dyadic}}} 
\sumstar_{\gcd(m_1,m_2) =1 }
a_{m_1} \overline{a_{m_2}}
\sum_{\gcd(\fa, r \mathcal{O}_k)=1}
\dot{W} \left(\frac{N(\fa)Q}{(m_1m_2)^{\phi(r)}} \right) \chi_{(m_1)}(\fa) \overline{\chi_{(m_2)}(\fa)}\\
&+
\frac{Q^{1-A}}{M^{\phi(r) - 2A\phi(r)}} \sum_{\substack{K > M^{2\phi(r)}/Q\\ \text{dyadic}}} K^{-A} B_2(K,M).
\end{align*}
(The claimed bound will use $A=2+\epsilon$.)

To deal with the first term, we separate the variables by applying the inverse Mellin transform (Lemma \ref{mellin}) for some $\sigma>0$
\[
\leq
Q
\sum_{\gcd(\fa, r \mathcal{O}_k)=1}
\sumstar_{\gcd(m_1,m_2) =1 }
\frac{a_{m_1} \overline{a_{m_2}}}{(m_1m_2)^{\phi(r)/2}}
 \chi_{(m_1)}(\fa) \overline{\chi_{(m_2)}(\fa)}
\int_{\sigma-i\infty}^{\sigma+i \infty} \dot{W}_{+}(s)  \left(\frac{N(\fa)Q}{(m_1m_2)^{\phi(r)}} \right)^{-s} \mathrm{d} s.
\]
The gcd condition is removed using M{\"o}bius inversion to get the upper bound
\[
Q
\sum_{\gcd(\fa, r \mathcal{O}_k)=1}
\sum_{d \leq 2M}\mu(d) d^{-\phi(r)}
\sumstar_{m_i \sim M/d }
\frac{a_{dm_1} \overline{a_{dm_2}}}{(m_1m_2)^{\phi(r)/2}}
 \chi_{(dm_1)}(\fa) \overline{\chi_{(dm_2)}(\fa)}
\int_{\sigma-i\infty}^{\sigma+i\infty} \dot{W}_{+}(s)  \left(\frac{N(\fa)Q}{(d^2m_1m_2)^{\phi(r)}} \right)^{-s} \mathrm{d} s.
\]
Taking absolute values and noting that $\vert n^s \vert = n^{\text{Re}(s)}$, we have the upper bound
\begin{align*}
\ll
\frac{Q^{1-\sigma}}{M^{1-4\sigma}}
\int_{-\infty}^{\infty}
\left \vert \dot{W}_{+}(\sigma + i t) \right \vert
\sum_{\gcd(\fa, r \mathcal{O}_k)=1}
N(\fa)^{-\sigma}
\sum_{d \leq 2M} d^{\phi(r)(2\sigma-1)}&
\left \vert
\sumstar_{m_1,m_2 \sim M/d} 
a_{dm_1} \overline{a_{dm_2}}
 \chi_{(dm_1)}(\fa) \overline{\chi_{(dm_2)}(\fa)}
\right \vert \mathrm{d} t.
\end{align*}
Setting $\sigma = \frac{\epsilon}{100}$ say and applying Cauchy--Schwarz to the $K$ sum, we get the bound \eqref{eq:lem6}.

For the third, observe that in the $B_2$ sum we have dropped the condition that every prime ideal dividing $\fa$ lies above a completely split rational prime. Hence $\fa$ may be divisible by some prime ideals lying above rational primes with non-trivial residue degree. Furthermore $\fa$ could be divisible by squares of ideals now. Accordingly we write $\fa$ as the product of a factor with square-free norm and one with square full norm
\[
B_2(Q,M)\leq \sum_{\substack{N(\mathfrak{c}) \leq 2Q \\ N(\mathfrak{c}) \text{ square-full}}} \sumprime_{N(\fb) \leq Q/N(\mathfrak{c})}
\mu^2(\fb \mathfrak{c}) \left \vert \sumstar_{m \sim M} \mu^2(m) a_m \chi_{(m)}(\fb \mathfrak{c}) \right \vert^2.
\]
We split the range of the $N(\mathfrak{c})$ sum dyadically so that 
\[
B_2(Q,M)
\ll
\log (2Q) \sup \limits_{1 \leq X \leq Q} 
\sum_{\substack{Q/X \leq N(\mathfrak{c}) \leq 2Q/X \\ N(\mathfrak{c}) \text{ square-full}}} \sumprime_{X/2^{\phi(r)} \leq N(\fb) \leq 2X}
\mu^2(\fb) \left \vert \sumstar_{m \sim M} \mu^2(m) a_m \chi_{(m)}(\fb \mathfrak{c}) \right \vert^2
\]
Therefore
\[
B_2(Q,M)
\ll
\log (2Q) \sup_X Q^{1/2}X^{-1/2} \left(B_1(X/2^{\phi(r)},M) + \ldots + B_1(X,M) \right).
\]
The claimed bound now follows from \eqref{eq:lem3}.

It just remains to prove \eqref{eq:lem2}.
We introduce the dual norm 
\[
C_2'(Q,M) :=  \sup_{(a_m) \not \equiv 0} \vert \vert a_m \vert \vert^{-2} \sumprime_{N(\fa) \sim Q} \mu^2(\fa) 
\left\vert \sum_{m \sim M}  a_m  \chi_{(m)}(\fa) \right\vert^2.
\] By duality, Lemma \ref{lem: duality principle}, we have $C_2(M,Q) = C_2'(Q,M)$. Assume that $(a_m)$ is a sequence attaining the supremum. Then, by applying H{\"o}lder's inequality, we have
\[
C_2'(Q,M) \ll \vert \vert a_m \vert \vert^{-2} Q^{1- 1/\nu}
\left( \sumprime_{N(\fa) \sim Q} \mu^2(\fa) \left \vert \sum_{M^{\nu} < m \leq (2M)^{\nu}} c_m \chi_{\fa}((m)) \right \vert^2 \right)^{1/\nu},
\]
where \(c_m := \sum_{m_1 \dots m_{\nu} = m} a_{m_1}\dots a_{m_{\nu}}.\)
We now break the $m$-sum into dyadic segments
\[
\left( \sumprime_{N(\fa) \sim Q} \mu^2(\fa)\left \vert \sum_{j=0}^{\nu-1} \sum_{2^jM^{\nu} < m \leq 2^{j+1}M^{\nu}} c_m \chi_{\fa}((m)) \right \vert^2 \right)^{1/\nu}\] and bound this by 
\[
\ll
\sum_{j=0}^{\nu-1}\left( \sumprime_{N(\fa) \sim Q}\mu^2(\fa)\left(\sum_{2^jM^{\nu} < m \leq 2^{j+1}M^{\nu}}\left \vert  c_m  \right \vert^2  \right) \vert \vert c_m \vert \vert^{-2}\sum_{2^jM^{\nu} < m \leq 2^{j+1}M^{\nu}}\left \vert  c_m \chi_{\fa}((m)) \right \vert^2 \right)^{1/\nu}.
\]
By the definition of $C_2'$, this suffices for the claimed bound.
\end{proof}
\begin{proof}[Proof of Theorem \ref{largesieve}] We now turn our attention to proving Theorem \ref{largesieve}.
Firstly, by the duality principle, we have $B_1(Q,M) = C_1(M,Q)$. Moreover, by the positivity of the square, we get that $C_1(M,Q) \leq C_2(M,Q)$. This, combined with \eqref{eq:lem1}, is already enough to prove the first term in \eqref{eq:lsconst}. For the next bound, we use \eqref{eq:lem2} with $\nu=2$ to see that
\[
C_2 \ll M^{\epsilon}Q^{1/2}\left[ C_2(M^2,Q)^{1/2} + C_2(2M^2,Q)^{1/2}\right].
\]
Applying \eqref{eq:lem1} yields the desired bound. Similarly, the next claimed bound in \eqref{eq:lsconst} follows from the case $\nu = 3$ of \eqref{eq:lem2}. We note that no further improvement is attained by taking larger values of $\nu$.
Setting $Q=X$ in \eqref{eq: lem4} and applying the bound that we just achieved, we have that
\[
B_2(Q,M) \ll (QM)^{\epsilon} \left(Q^{2/3}M+ Q^{4/3}\right).
\]
%
Using this and \eqref{eq:lem6} gives
\begin{align*}
B_3(Q,M) 
&\ll
Q + (QM)^{\epsilon}\frac{Q}{M^{\phi(r)}} \max \limits_{1 \leq K \leq (QM)^{\epsilon}M^{2\phi(r)} Q^{-1}} B_2(K,M)
+\frac{M^{3\phi(r)}}{Q} \sum_{K >M^{2 \phi(r)}/Q} K^{-2}B_2(K,M),\\
&\ll
Q + (QM)^{\epsilon}\frac{Q}{M^{\phi(r)}} \max \limits_{1 \leq K \leq (QM)^{\epsilon}M^{2\phi(r)} Q^{-1}} (K^{2/3}M + K^{4/3})
\\&\phantom{\ll} +\frac{M^{3\phi(r)}}{Q} \sum_{K >M^{2 \phi(r)}/Q} (K^{-4/3}M + K^{-2/3}),\\
&\ll
Q + (QM)^{\epsilon}(Q^{1/3}M^{\phi(r)/3 + 1} + Q^{-1/3}M^{5\phi(r)/3}).
\end{align*}
By combining\eqref{eq:lem3} and \eqref{eq:lem5} with this, we have
\[
B_1(Q,M) \ll Q + (QM)^{\epsilon}(Q^{1/3}M^{\phi(r)/3 + 1} + Q^{-1/3}M^{5\phi(r)/3}).
\]
Finally we may replace $Q$ by $Q^{1+\epsilon} + M^{2 \phi(r) - 3/2}$ and use the increasing property to get the claimed bound.
\end{proof}
\section{Interlude: Application of the large sieve} \label{sec:interlude}
Our aim in this section is to establish the following second moment estimate.
\[
\sum_{q \leq Q}\, \sumstar_{\substack{ \chi \bmod q \\ \chi^r = \chi_0\\ \chi \neq \chi_0}}
\vert L(1/2 + it, \chi) \vert^2 \ll  Q^{\frac{7}{6}+\epsilon}(1 + \vert t \vert)^{{\frac{1}{2}+\epsilon}} +  Q^{\frac{4}{3}+\epsilon}.
\]


The key ingredient is the approximate functional equation for Dirichlet $L$-functions (see e.g. \cite[Thm 5.3]{IK}).
\begin{lemma}[Approximate Functional Equation]
Let $\chi$ be a primitive Dirichlet character of conductor $q$. Let
\[
V_{\alpha}(x) = \frac{1}{2 \pi i} \int_{1-i\infty}^{1+i\infty} \frac{G(s)}{s}g_{\alpha}(s) x^{-s} \mathrm{d}s, \quad{} \text{ where }
g_{\alpha}(s) = \pi^{-s/2}\frac{\Gamma\left(\frac{\frac{\delta}{2}+\alpha + s}{2}\right)}{\Gamma\left(\frac{\frac{\delta}{2}+\alpha}{2}\right)}.
\]
Here $\delta =1$ if $\chi$ is even and 3 if $\chi$ is odd.
Furthermore, let $\epsilon(\chi) = i^{-1}q^{-1/2}\tau(\chi)$ and set 
\[
X_{\alpha} = \left(\frac{q}{\pi}\right)^{- \alpha} \frac{\Gamma\left(\frac{\frac{\delta}{2}-\alpha}{2}\right)}{\Gamma\left(\frac{\frac{\delta}{2}+\alpha}{2}\right)}.
\]
Then for any $\vert \Re(\alpha) \vert < \frac{1}{2}$, we have
\begin{equation}\label{eq:func}
L\left(\frac{1}{2} + \alpha, \chi\right)
=
\sum_{m=1}^{\infty} \frac{\chi(m)}{m^{\frac{1}{2}+\alpha}} V_{\alpha}\left(\frac{m}{\sqrt{q}}\right)
+
\epsilon(\chi)X_{\alpha} \sum_{m=1}^{\infty} \frac{\overline{\chi(m)}}{m^{\frac{1}{2}- \alpha}} V_{-\alpha}\left(\frac{m}{\sqrt{q}}\right).
\end{equation}
\end{lemma}

Using the inequality $\vert a + b\vert^2 \leq 2 \vert a \vert^2 + 2 \vert b \vert^2$, we have
\[
\sum_{ q \leq Q} \sumstar_{\substack{ \chi \bmod q\\ \text{ord}(\chi) = r}} \left\vert L\left(\frac{1}{2} + it, \chi\right) \right\vert^2
\ll\sum_{q \leq Q} \sumstar_{\substack{ \chi \bmod q\\ \text{ord}(\chi) = r}}
\left \vert \sum_{m =1}^{\infty} \frac{\chi(m)}{m^{\frac{1}{2} + it}} V_{it}\left(\frac{m}{\sqrt{q}}\right)\right \vert^2.
\]
(Note that the term involving $V_{-it}$ contributes an identical amount.)
Using standard bounds, we have that $V_{it}(x) \ll_A (1 + x(1+\vert t \vert )^{-1/2})^{-A}$ so that we may truncate the $m$ sum at $M := (Q(1+\vert t \vert))^{1/2+\epsilon}$ upto an acceptable error term. 
We break both the $q$ and $m$ sum into dyadic intervals. The weight $V_{it}$ may be removed by an application of the inverse Mellin transform (Lemma \ref{mellin}). 
This gives an upper bound of the form
\[
\ll
\frac{Q^{\sigma}}{M^{2\sigma}} \int_{-\infty}^{\infty} \vert V_{it,+}(\sigma + it) \vert
\sum_{q \leq Q} \sum_{\substack{ \chi \bmod q \\ \text{ord}(\chi) = r}}
\left \vert \sum_{M' \leq m < 2M'} \frac{\chi(m)}{m^{1/2 + it}} \right \vert^2
\mathrm{d}t.
\]

We may write $m = d^2 n$ for $n$ squarefree then applying the Cauchy--Schwarz inequality one is left to evaluate the sums
\[
\sum_{d \leq \sqrt{2M'}}\frac{1}{d}
\sum_{Q' \leq q \leq 2 Q'} \sumstar_{\substack{\chi \bmod q \\ \text{ord}(\chi) =r}} 
\left \vert \sum_{M'/d^2 \leq m \leq 2M'/d^2}\frac{\chi(m)}{m^{\frac{1}{2} + it}}  \right \vert^2.
\]
To do this we apply Theorem \ref{largesieve}, specifically the bound $\ll Q^{2/3}M + Q^{4/3}$.



\section{The major arcs} \label{section: major arc}
Let $r \in \NN$ be a prime. Our starting point in the proof of Theorem \ref{thm: average} is the identity 
\[
\sum_{\substack{n \le x \\ n \equiv n_0 \bmod{M_0} }} \Lambda(n^r + k)
=
\int_0^1 \sum_{m \leq x^r + k}\Lambda(m)e(\alpha m) \sum_{\substack{n \le x \\ n \equiv n_0 \bmod{M_0} }}e(- \alpha(n^r + k)) \mathrm{d}\alpha
.\]
Note that if $\gcd(n_0^r +k,M_0) \neq 1$ then there is no way that the left hand side is non-zero, so henceforth we assume they are coprime.
As is typical in circle method problems, we will partition the interval $[0,1]$ into major arcs $\mathfrak{M}$ and minor arcs $\mathfrak{m}$.
For $Q_1 := (\log x)^{c_1}$ for some fixed $c_1>0$, we set
$$
\mathfrak{M} =  \bigcup_{q \le Q_1} \bigcup_{\substack{a=1 \\ (a,q)=1}}^{q} I_{a,q},
$$
where
$$
I_{a,q} := \left [\frac{a}{q} - \frac{1}{qQ_2} , \frac{a}{q} + \frac{1}{qQ_2} \right ], 
$$
where $Q_2:=x^{r-1-\varepsilon}$. When $x$ is sufficiently large the inequality $Q_2 > Q_1$ holds, implying that the intervals $I_{a,q}$ with $q \le Q_1$ become disjoint.

\begin{remark}
The choice of major arcs is slightly unconventional here. Detecting primes in the major arcs is a problem comparable to evaluating character sums over primes in short intervals. The smaller one makes $Q_2$, the smaller these intervals become and the harder the problem. However, if $Q_2$ has size close to $x^r$, then it is not possible to get the necessary cancellation in the minor arcs.
\end{remark}

We write $\alpha \in \mathfrak{M}$ as
$\alpha = \frac{a}{q} + \beta$ for some $a, q \in \mathbb{Z}$ with $\gcd(a,q)=1$ and with $|\beta| \le \frac{1}{qQ_2}$.

Let 
\begin{align*}
S_1(\alpha) & := \sum_{m \le x^r+k} \Lambda(m)e(\alpha m),\\
S_2(\alpha)  & := \sum_{\substack{n \le x \\ n \equiv n_0 \bmod{M_0} }} e  (-\alpha n^r  ).
\end{align*}

We will first estimate $S_1(\alpha)$. For convenience, we set
\begin{equation} \label{eq: z}
z:= x^r +k.
\end{equation}
We have
\begin{align*}
S_{1}(\alpha) & = \sum_{m \le z} \Lambda(m)e(am/q)e(\beta m) \nonumber \\
& = \sum_{\substack{m \le z \\ (m,q)=1}} \Lambda(m)e(a m/q)e(\beta m) +  \sum_{\substack{m \le z \\ (m,q)> 1}} \Lambda(m)e(a m/q)e(\beta m). 
\end{align*}
Let $q = \prod_{i=1}^t p_i^{\delta_i}$ be the prime factorisation of $q$.
If $ (m,q)> 1$, then $\Lambda(m) \neq 0$ if and only if $m = p_i^{\alpha_i}$ for some $i \in \{1, ..., t\}$ and some $\alpha_i > 0$. Hence,
\begin{align*}
\sum_{\substack{m \le z \\ (m,q)> 1}} \Lambda(m)e(a m/q)e(\beta m) = & \sum_{i=1}^t \sum_{\alpha_i = 1}^{\infty} \sum_{p_i^{\alpha_i} < z} \log(p_i) e(a p_i^{\alpha_i}/q)e(\beta  p_i^{\alpha_i}) \\
\leq & \log z  \sum_{i=1}^t \sum_{\alpha_i = 1}^{\log_{p_i}(z)}  1.
\end{align*}
If $p_i \geq 3$, then $\log_{p_i}(z) \leq \log z$. Since $\log_{2} z = \log z/\log 2$, we have
\[
\sum_{\substack{m \le z \\ (m,q)> 1}} \Lambda(m)e(a m/q)e(\beta m)
\ll \log q \log^2 z.
\]

Hence, 
\begin{align*}
S_{1}(\alpha) & = \sum_{\substack{m \le z \\ (m,q)=1}} \Lambda(m)e(a m/q)e(\beta m) +  O(\log q \log^2 z). 
\end{align*}
We may rewrite the main term as
\begin{align*}
\sum_{\substack{m \le z \\ (m,q)=1}} \Lambda(m)e(a m/q)e(\beta m) = \frac{1}{\varphi(q)} \sum_{\chi  \bmod q} \tau(\overline{\chi}) \chi(a) \sum_{\substack{m \le z\\  (m,q)=1}} \chi(m) \Lambda(m) e(\beta m),
\end{align*}
where $
\tau(\chi) := \sum_{n=1}^{q} \chi(n) e(n/q)
$
is the Gauss sum.

The largest contribution will occur when the character $\chi$ is the principal character $\chi_0 \Mod q$ so we separate this case out
\begin{align*}
&  \frac{1}{\varphi(q)} \sum_{\chi  \bmod q} \tau(\overline{\chi}) \chi(a)  \sum_{\substack{m \le z\\  (m,q)=1}} \chi(m) \Lambda(m) e(\beta m)\\
= &  \frac{1}{\varphi(q)} \tau(\overline{\chi_0}) \chi_0(a) \sum_{\substack{m \le z\\  (m,q)=1}} \chi_0(m) \Lambda(m) e(\beta m) + \frac{1}{\varphi(q)} \sum_{\substack{\chi  \bmod q  \\ \chi \neq \chi_0 }} \tau(\overline{\chi}) \chi(a)  \sum_{\substack{m \le z\\  (m,q)=1}}\chi(m) \Lambda(m) e(\beta m).
\end{align*}
Further since $\tau(\overline{\chi_0}) = \mu(q)$ and $\gcd(a,q) =1$, we have
\begin{align*}
 & \frac{1}{\varphi(q)}  \tau(\overline{\chi_0}) \chi_0(a)  \sum_{\substack{m \le z\\  (m,q)=1}} \chi_0(m) \Lambda(m) e(\beta m)\\
=&   \frac{\mu(q)}{\varphi(q)} \sum_{\substack{m \le z}} \Lambda(m) e(\beta m) + O(\log q \log^2 z)\\
=&   \frac{\mu(q)}{\varphi(q)} \left  (\sum_{\substack{m \le z}}e(\beta m) + \sum_{\substack{m \le z}} (\Lambda(m)-1)e(\beta m) \right  ) + O(\log q \log^2 z).\\
\end{align*}
It follows that
\begin{align*}
&  \frac{1}{\varphi(q)} \sum_{\chi  \bmod q} \tau(\overline{\chi}) \chi(a)  \sum_{\substack{m \le z\\  (m,q)=1}} \chi(m) \Lambda(m) e(\beta m)\\
= &  \underbrace{\frac{\mu(q)}{\varphi(q)} \sum_{\substack{m \le z}} e(\beta m)}_{=:T_1(\alpha)} + \underbrace{\frac{1}{\varphi(q)} \sum_{\substack{\chi \bmod q   }} \tau(\overline{\chi}) \chi(a) \sum_{\substack{m \le z\\ \gcd(m,q) =1}}^{\#} \chi(m) \Lambda(m) e(\beta m)}_{=:E_1(\alpha)} + O(\log q \log^2 z),
\end{align*}
where the $\#$ over the summation symbol means that if $\chi=\chi_0$, then $\chi(n) \Lambda(n)$ is replaced by $\Lambda(n)-1$.
Hence,
$$
 S_1(\alpha)= T_1(\alpha) + E_1(\alpha) +O \left (\log q \log^2 z\right ).
$$
We write
\[
S_2(\alpha) = \sum_{\substack{n \leq x \\ n \equiv n_0 \bmod{M_0}}} e\left(-\frac{a}{q}n^r\right) e(- \beta n^r)
=
\sum_{\substack{b \bmod{[M_0,q]}\\b \equiv n_0 \bmod{M_0}}} e\left(-\frac{a}{q}b^r\right) \sum_{\substack{n \leq x\\ n \equiv b \bmod{[M_0,q]}}} e (- \beta n^r) .
\]

\section{Computing the singular series} 
The main term contributing in the major arc is
\begin{align*}
& \int_{\mathfrak{M}} T_1(\alpha) S_2(\alpha) e(-\alpha k) \ \mathrm{d}  \alpha \\
=&  \sum_{q \le Q_1 } \frac{\mu(q)}{\varphi(q)} \sum_{\substack{a   \bmod q \\ (a,q)=1}} e \left (\frac{-ak}{q} \right )
\sum_{\substack{b \bmod{[M_0,q]}\\b \equiv n_0 \bmod{M_0} }} e\left(-\frac{a}{q}b^r\right)
 \int_{|\beta| < \frac{1}{qQ_2}} \Theta_{d,q}(\beta) \ \mathrm{d} \beta,
\end{align*}
where
$$
 \Theta_{d,q}(\beta) := e(-\beta k)  \sum_{\substack{n \leq x \\ n \equiv b \bmod{[M_0,q]}}} e \left (-\beta n^r \right ) \sum_{m \le z} e \left (-\beta m \right ). 
$$
By the linear case of Lemma \ref{lem: weyl}, extending the integral over $\beta$ to all of $[0,1]$ incurs an error of size at most 
\[
\int_{\frac{1}{qQ_2}}^{\frac{1}{2}} \frac{1}{\beta} \left\vert  \sum_{\substack{n \leq x \\ n \equiv b \bmod{[M_0,q]} }} e \left (-\beta n^r \right ) e(-k\beta)\right\vert \mathrm{d}\beta.
\]

We apply H\"older's inequality  to bound this error by
\[
\begin{array}{l}
\ll \left(\int_{\frac{1}{qQ_2}}^{\frac{1}{2}} \frac{1}{\beta^{\frac{2^r}{2^r-1}}} \mathrm{d}\beta \right)^{\frac{2^r-1}{2^r}}
\left(\int_{0}^{1} \left \vert \sum_{\substack{n \leq x \\ n \equiv b \bmod{[M_0,q]}}} e \left (-\beta n^r \right) e(-k\beta) \right\vert^{2^r} \mathrm{d}\beta\right)^{\frac{1}{2^r}}\\
\ll \left((qQ_2)^{\frac{2^r}{2^r-1}-1}\right)^{\frac{2^r-1}{2^r}}
\left(\int_{0}^{1} \left \vert  \sum_{\substack{n \leq x \\ n \equiv b \bmod{[M_0,q]}}} e \left (-\beta n^r \right)  \right\vert^{2^r} \mathrm{d}\beta\right)^{\frac{1}{2^r}}\\
\ll  (qQ_2)^{\frac{1}{2^r}}
\left(\int_{0}^{1} \left \vert \sum_{\substack{n \leq x \\ n \equiv b \bmod{[M_0,q]}}} e \left (\beta n^r \right)  \right\vert^{2^r} \mathrm{d}\beta\right)^{\frac{1}{2^r}},
\end{array}
\]
where the last inequality follows by a change of variables.

By Hua's lemma with congruences (see Lemma \ref{hua}), we can choose any $0<\epsilon'' < \epsilon$ and obtain $\int_{0}^{1} \left \vert \sum_{\substack{n \leq x \\ n \equiv b \bmod{[M_0,q]}}} e \left (\beta n^r \right)  \right\vert^{2^r} \mathrm{d}\beta \ll \left(\frac{x}{[M_0, q]}\right)^{2^r-r + \epsilon''} $. Hence, the error above can be bound by
\[
\begin{array}{l}
\ll  (qQ_2)^{\frac{1}{2^r}} \left( \left(\frac{x}{[M_0, q]}\right)^{2^r-r + \epsilon''}\right)^{\frac{1}{2^r}}\\
 \ll   \left(\frac{1}{[M_0, q]}\right)^{1 - \frac{r}{2^r} + \frac{\epsilon}{2^r}} \left((\log x)^{c_1}  x^{2^r  -1 + \epsilon'' - \epsilon}\right)^{\frac{1}{2^r}}\\
\ll \left(\frac{1}{[M_0, q]}\right)^{1 - \frac{r}{2^r} + \frac{\epsilon}{2^r}}  (\log x)^{\frac{c_1}{2^r}} x^{1 - \epsilon'},\\
\end{array}
\]
where $\epsilon' := \frac{-\epsilon'' + \epsilon + 1}{2^r}>0$.

Meanwhile, by orthogonality, we have  
\[
\int_0^1  \Theta_{d,q}(\beta) \mathrm{d}\beta
=
\sum_{\substack{n \leq x  \\ n \equiv b \bmod{[M_0,q]}}} 1
=
\frac{x}{[M_0,q]}+ O \left(1\right).
\]

Therefore
\begin{align*}
& \int_{\mathfrak{M}} T_1(\alpha) S_2(\alpha) e(-k\alpha) \ \mathrm{d} \alpha \\
 =  & \sum_{q \le Q_1 } \frac{\mu(q)}{\varphi(q)} \sum_{\substack{a \bmod q \\ (a,q)=1}} e \left (\frac{-ak}{q} \right )\sum_{\substack{b \bmod{[M_0,q]}\\b \equiv n_0 \bmod{M_0} }} e\left(-\frac{a}{q}b^r\right) \left \{ \frac{x}{qM_0}\gcd(q,M_0) + O \left (   \frac{ (\log x)^{\frac{c_1}{2^r}} x^{1 - \epsilon'} }{[M_0, q]^{1 - \frac{r}{2^r} + \frac{\epsilon}{2^r}}} \right )  \right \}\\
= & \frac{x}{M_0}\sum_{q \le Q_1 } \frac{\mu(q)\gcd(q,M_0)}{\varphi(q)q} \sum_{\substack{a  \bmod q \\ (a,q)=1}} e \left (\frac{-ak}{q} \right ) \sum_{\substack{b \bmod{[M_0,q]}\\ b \equiv n_0 \bmod{M_0}}} e \left (\frac{-a b^r}{q} \right  )  + O \left ( (\log x)^{c_2} x^{1 - \epsilon'}  \right ),
\end{align*}
for some suitable $c_2>0$.

 We define
\[
\Sigma(q)   := \sum_{\substack{a \bmod q \\ (a,q)=1}} e \left (\frac{-ak}{q} \right )\sum_{\substack{b \bmod{[M_0,q]}\\ b \equiv n_0 \bmod{M_0}}} e \left (\frac{-a b^r}{q} \right  )
=
\sum_{\substack{c \bmod{q/(q,M_0)}}}  \sum_{\substack{a \bmod q \\ (a,q)=1}} e \left (\frac{-a((n_0 + c M_0)^r + k)}{q} \right) .\]

Firstly, we note that $\Sigma(q)$ is a multiplicative function of $q$.
Indeed for $\gcd(q_1,q_2) = 1$, let $q_i' := q_i/\gcd(q_i,M_0)$. Then by the Chinese Remainder theorem, we may write $\Sigma(q_1q_2)$ as
\[
\Sigma(q_1q_2)
=
\mathop{\sum_{c_1 = 1}^{q_1'}
\sum_{c_2 =1}^{q_2'}}
\mathop{
\sum_{\substack{a_1 \Mod q_1 }}
\sum_{\substack{a_2 \Mod q_2 }}}_{ \gcd(a_1q_2 + a_2q_1, q_1q_2) =1 }
e \left( - \frac{(a_1q_2+a_2q_1)\left(((c_1q_2' + c_2q_1')M_0 + n_0)^r + k\right)}{q_1 q_2}\right)
=
\Sigma(q_1) \Sigma(q_2).
\]
Therefore, $\Sigma(q)$ is determined by prime power values of $q$, but since $q$ is squarefree, we only need to understand when $q=p$ is a prime. 
In this case, the value of the inner Ramanujan sum is
\[
 \sum_{\substack{a \bmod p \\ (a,p)=1}} e \left (\frac{-a((cM_0+n_0)^r + k)}{p} \right) 
=
\begin{cases}
p-1 \text{ if } (cM_0 + n_0)^r + k \equiv 0 \Mod p,\\
-1 \text{ otherwise}.
\end{cases}
\]
Since $\gcd(n_0^r +k, M_0) = 1$, we have that $\Sigma(p) = -1$ for all $p \mid M_0$.
At all other primes we may make a linear change of variables so that
\[\Sigma(p) 
 = (p-1)n_{k,p}-(p-n_{k,p}) 
 = p(n_{k,p} -1 ),\]
where $n_{k,p}$ is the number of $\mathbb{Z}/p\mathbb{Z}$-solutions to \(
u^r + k \equiv 0  \bmod{p}.\)
It follows that, for some fixed constant $c_2 >0$, we have
\begin{align}
 \int_{\mathfrak{M}} T_1(\alpha) S_2(\alpha) e(-k\alpha) \ \mathrm{d} \alpha &  = \frac{x}{M_0} \sum_{q \le Q_1 } \frac{\mu(q)\gcd(q,M_0)}{\varphi(q)q} \Sigma(q) +   O \left ((\log x)^{c_2} x^{1 - \epsilon'}  \right ) \nonumber \\
 & = \mathfrak{S}_{n_0, M_0}(k) x+ O \left ( x |\Psi(k)| +(\log x)^{c_2} x^{1 - \epsilon'}  \right ), \label{eq: marjor arc}
\end{align}
where
\begin{equation} \label{psik}
\Psi(k):=\sum_{q >Q_1} \frac{\mu(q)}{\varphi(q)} \prod_{p|q}(n_{k,p}-1).
\end{equation}

Note that the map $x \mapsto x^r$ is a bijection on $\ZZ/p\ZZ$ if and only if $\gcd(r, p-1)=1$, and in this case $n_{k,p} = 1$. Note also that, since $r$ is assumed to be prime, we have $\gcd(r, p-1)=1$ if and only if $ p \not\equiv 1 \bmod{r}$.
If $p \equiv 1 \bmod{r}$, then it is well known that
\[n_{p,k} = \sum_{\substack{\chi  \bmod{p} \\ \chi^r = \chi_0}} \chi(-k ) =  \sum_{\substack{\chi  \bmod{p} \\ \chi^r = \chi_0\\ \chi \neq \chi_0}} \chi(-k ) + 1.\]
We deduce that
\begin{equation}\label{blahblah}
n_{k,p}-1=
\begin{cases}
0 & \mbox{if $p \not\equiv 1 \bmod{r}$,} \\
 \sum \limits_{\substack{ \chi \bmod p \\ \chi^r = \chi_0 \\ \chi \neq \chi_0}} \chi(-k) & \mbox{if $p\equiv 1 \bmod{r}$.} \\
\end{cases}
\end{equation}

Hence, we can assume in \eqref{psik} that $q$ is such that $p \equiv 1 \bmod{r}$ for all $p |q$.


\section{Bounding the second moment of $\Psi(k)$} \label{se tion: second moment of Psi}

We now study the second moment of \eqref{psik}.
First, we partition the second moment into three pieces
$$
\sum_{\substack{k \le y}} |\Psi(k)|^2 \ll \Psi_1 + \Psi_2 + \Psi_3,
$$
where
\[
\begin{array}{lll}
\Psi_1 & := & \sum_{k \le y} \left | \sum_{Q_1 < q \le U }  \frac{\mu(q)}{\varphi(q)} \prod_{p|q}(n_{k,p}-1 ) \right|^2, \\
\Psi_2 &:=  &\sum_{\substack{k \le y }}  \left | \sum_{U < q \le 2^\nu U }  \frac{\mu(q)}{\varphi(q)} \prod_{p|q}( n_{k,p}-1) \right|^2, \\
\Psi_3 &:=  &\sum_{\substack{k \le y }}  \left | \sum_{ 2^\nu U < q }  \frac{\mu(q)}{\varphi(q)} \prod_{p|q}( n_{k,p}-1  )  \right|^2,
\end{array}\]
and $U,\nu$ are parameters to be chosen later (see the end of this section).

 We deal with $\Psi_1$ first. Expanding the square and using the fact that a degree $r$ polynomial has at most $r$ roots in $\FF_p$, we obtain 
\begin{align} \label{rando1}
\Psi_1 \le &   \hspace{2mm} y \sum_{Q_1 < q \le U } \frac{\mu^2(q)}{\varphi(q)^2} (r-1)^{2\omega(q)}  + \sum_{ \substack{ Q_1 < q_1, q_2 \le U \\ q_1 \neq q_2 }} \frac{\mu(q_1) \mu(q_2)}{\varphi(q_1) \varphi(q_2) } \sum_{k \le y} P_{q_1,q_2}(k),
\end{align}
where
\begin{align*}
& P_{q_1,q_2}(k)   := \prod_{p|q_1} ( n_{k,p}-1) \prod_{p|q_2}( n_{k,p}-1).
\end{align*}
The first term in the right-hand side of \eqref{rando1} can be bounded by
$$
y \sum_{Q_1 < q \le U} \frac{(r-1)^{2\omega(q)} (\log \log  10q)^2}{q^2} \ll \frac{y}{Q_1^{1-\varepsilon}},
$$
where $\omega(q)$ is the number of distinct primes dividing $q$, 
via the well-known bounds 
\begin{equation*} 
\frac{q}{\log \log 10q} \ll \varphi(q),
\end{equation*}
and
$$
 \omega(q) \ll \frac{\log q}{\log \log q}.
$$

By \eqref{blahblah}, for any $q_1 \neq q_2$,  we have that $P_{q_1,q_2}(k) = 0$ unless $p_1 \equiv 1 \bmod{r}$ for all $p_1 | q_1$ and $p_2 \equiv 1 \bmod{r}$ for all $p_2 | q_2$. Expanding $P_{q_1,q_2}(k)$ thus yields
\begin{equation}\label{Pq1q2}
P_{q_1,q_2}(k) = \sumstar_{\substack{\chi_1 \bmod{q_1}\\ \chi_1^r = \chi_0\\ \chi_1 \neq \chi_0 }} \chi_1(-k)\sumstar_{\substack{\chi_2 \bmod{q_2}\\ \chi_2^r = \chi_0 \\ \chi_2 \neq \chi_0 }} \chi_2(-k)
\end{equation}
where $q_1$ and  $q_2$ are such that $p_1 \equiv 1 \bmod{r}$ for all $p_1 | q_1$ and $p_2 \equiv 1 \bmod{r}$ for all $p_2 | q_2$.
Note that  the resulting characters $\chi_1 \chi_2$ after expanding \eqref{Pq1q2}
 are of modulus $q_1 q_2$ and cannot be principal,  since $\chi_1$ and $\chi_2$ are both primitive and $q_1 \neq q_2$. 
Hence, by P{\'o}lya--Vinogradov (Lemma \ref{lem: polya-vinogradov}), the second term in \eqref{rando1} can be bounded by
\begin{align*}
\sum_{\substack{Q_1 < q_1,q_2 \le U \\ q_1 \neq q_2 }} \frac{(q_1q_2)^{1/2} \log (q_1q_2) \phi(r)^{\omega(q_1) + \omega(q_2) }}{\varphi(q_1) \varphi(q_2)} & \ll (\log U) \left (\sum_{Q_1 < q \le U}  \frac{q^{1/2} \phi(r)^{\omega(q)}}{\varphi(q)} \right )^2\\
&  \ll U^{1+\varepsilon}.
\end{align*}
Hence, we get
\begin{equation} \label{eq: Psi_1 bound}
\Psi_1 \ll \frac{y}{Q_1^{1-\varepsilon}}  + U^{1+\varepsilon}.
\end{equation}


Let us now turn to bounding $\Psi_2 $. 
We start by breaking the $q$-sum into dyadic intervals, so that
\begin{align*}
\Psi_2 &= 
 \sum_{\substack{k \le y}}  \left | \sum_{s=1}^{\nu}\sum_{2^sU < q \le 2^{s+1}U }  \frac{\mu(q)}{\varphi(q)} \sum_{\substack{\chi \bmod q \\ \chi^r = \chi_0}} \chi(-k) \right|^2\\
&\leq 
\nu
 \sum_{\substack{k \le y }}
\sum_{s=1}^{\nu}
 \left | \sum_{2^sU < q \le 2^{s+1}U }  \frac{\mu(q)}{\varphi(q)} \sum_{\substack{\chi \bmod q \\ \chi^r = \chi_0}} \chi(-k) \right|^2.
\end{align*}

For $R$ a suitable parameter, to be chosen later, and $s \leq R$, we aim to use the large sieve for order $r$ characters to bound
\[  \sum_{\substack{k \le y }}  \left | \sum_{2^sU < q \le 2^{s+1}U }  \frac{\mu(q)}{\varphi(q)}  \sum_{\substack{\chi \mod q \\ \chi^r = \chi_0}} \chi(-k) \right|^2.\]

Note that by Theorem~\ref{largesieve}  the dual version of this sum can be bounded by 
\begin{align*}
 & \sum_{2^sU < q \le 2^{s+1}U } \sum_{\substack{\chi  \bmod q \\ \chi^r = \chi_0}}  \left | \sum_{\substack{k \le y  }} a_k \chi(k) \right |^2 \\
 =&\sum_{2^sU < q \le 2^{s+1}U } \sum_{\substack{\chi  \bmod q \\ \chi^r = \chi_0}} \left |  \sum_{\ell^2 \le y} \sum_{\substack{m \le y /\ell^2}} \mu^2(m) a_{\ell^2 m} \chi(\ell^2) \chi(m) \right |^2  \\
 \ll & y^{1/2 }  \sum_{2^sU < q \le 2^{s+1}U } \sum_{\substack{\chi  \bmod q \\ \chi^r = \chi_0}}  \sum_{\ell^2 \le y}    \left |  \sum_{\substack{m \le y /\ell^2}} \mu^2(m) a_{\ell^2 m} \chi(m) \right |^2   \\
  \ll & y^{1/2}  (2^sU +y^{\frac{r+2}{3}}(2^sU)^{1/3}+y^{r-\frac{1}{2}})  \sum_{\substack{ k \le y  }}  |a_k|^2.
\end{align*}
Hence by the duality principle (Lemma~\ref{lem: duality principle}), we have
\begin{align*}
\sum_{\substack{k \le y  }}   &  \left | \sum_{2^sU < q \le 2^{s+1}U } \sum_{\substack{\chi  \bmod q \\ \chi^r = \chi_0}} \frac{\mu(q)}{\varphi(q)}\chi(k) \right |^2  \\
&  \ll y^{1/2 } (2^sUy)^{\epsilon} (2^sU +y^{\frac{r+2}{3}}(2^sU)^{1/3}+y^{r-\frac{1}{2}})   \sum_{2^sU <q \le 2^{s+1} U} \frac{1}{\varphi(q)^2}. \\
\end{align*}
But $  \sum_{2^sU <q \le 2^{s+1} U} \frac{1}{\varphi(q)^2} \ll (2^sU)^{\epsilon-1}$, and so 
\[\sum_{\substack{k \le y  }}     \left | \sum_{2^sU < q \le 2^{s+1}U } \sum_{\substack{\chi  \bmod q \\ \chi^r = \chi_0}} \frac{\mu(q)}{\varphi(q)}\chi(k) \right |^2 \ll y^{1/2}(2^sUy)^{\epsilon}  
(1 + y^{\frac{r+2}{3}}(2^sU)^{-2/3} + y^{r-\frac{1}{2}}(2^sU)^{-1} ) .\]
Taking $R=\lfloor \log_{2} ( y^{2\phi(r) +\epsilon}/U) \rfloor$  and summing over the $s \leq R$ yields
\begin{equation}\label{psi21}
\sum_{\substack{k \le y }}     \left | \sum_{U < q \le 2^RU } \sum_{\substack{\chi  \bmod q \\ \chi^r = \chi_0}} \frac{\mu(q)}{\varphi(q)}\chi(k) \right |^2  \ll y^{1/2 + \epsilon} .
\end{equation}

We now consider the range when $q \in [2^RU, 2^\nu U]$. 
Recall (Proposition \ref{correspondence}) that primitive characters of order $r$ and  of squarefree conductor $q$ coprime to $r$ can be realised as $r$-th power residue symbols $\left (\frac{m}{\mathfrak{a}} \right )_r$ for some square-free  $\mathfrak{a}\subset \mathcal{O}_{\mathbb{Q}(\zeta_r)}$, where every prime dividing $\mathfrak{a}$ lies above a completely split rational prime and $N(\mathfrak{a}) = q$. 
Therefore when $R< s \le \lfloor \nu+1 \rfloor$, it is enough to bound
\begin{equation*} \label{eq: large sieve}
 \sum_{\substack{m \in \mathcal{O}_{\mathbb{Q}(\zeta_r)} \\ N(m) \leq y^{\phi(r)} }} \left | \sum_{\substack{\mathfrak{a} \subset\mathcal{O}_{\mathbb{Q}(\zeta_r)} \\ 2^{s-1} < N(\mathfrak{a})=q \le 2^s U}} b_{\mathfrak{a}} \chi_{(m)}(\mathfrak{a})  \right |^2,
\end{equation*}
where $b_{\mathfrak{a}} := \mu(N(\mathfrak{a}))/\varphi(N(\mathfrak{a}))$, $\chi_{(m)}(\mathfrak{a}) := \left (\frac{m}{\mathfrak{a}} \right )_r$. 
To this we may apply the number field large sieve (see Lemma \ref{lem: large sieve number field}) and get
\begin{align*}
 \sum_{\substack{m \in \mathcal{O}_{\mathbb{Q}(\zeta_r)} \\ N(m) \leq y^{\phi(r)} }} \left | \sum_{\substack{\mathfrak{a} \subset \mathcal{O}_{\mathbb{Q}(\zeta_r)} \\ 2^{s} < N(\mathfrak{a})=q \le 2^{s+1} U}} b_{\mathfrak{a}} \chi_{(m)}(\mathfrak{a})  \right |^2
&\ll
(y^{2 \phi(r)} + 2^s U) \sum_{\substack{\mathfrak{a} \subset\mathcal{O}_{\mathbb{Q}(\zeta_r)} \\ 2^{s} < N(\mathfrak{a})=q \le 2^{s+1} U}} \frac{1}{\phi(N(\mathfrak{a}))}\\
& \ll (y^{2 \phi(r)} + 2^s U) \frac{\log\log(2^s U) \log(2^s U)}{2^s U}.
\end{align*}
Summing over the $s$ from $R$ to $\lfloor\nu +1\rfloor$, we get
\begin{align}
& \sum_{s = R}^{\lfloor \nu +1\rfloor}
 \sum_{\substack{m \in \mathcal{O}_{\mathbb{Q}(\zeta_r)} \\ N(m) \leq y^{\phi(r)} }} \left | \sum_{\substack{\mathfrak{a} \subset \mathcal{O}_{\mathbb{Q}(\zeta_r)} \\ 2^{s} < N(\mathfrak{a})=q \le 2^{s+1} U}} b_{\mathfrak{a}} \chi_{(m)}(\mathfrak{a})  \right |^2 \nonumber\\
\ll & y^{2 \phi(r)}(\nu + \log U)\log(\nu + \log U) \sum_{s = R}^{\lfloor \nu +1 \rfloor} \frac{1}{2^s U} 
+
\sum_{s = R}^{\lfloor \nu +1 \rfloor}(s+ \log U)\log(s + \log U)\nonumber\\
\ll &
\frac{ y^{2 \phi(r)}(\nu + \log U)\log(\nu + \log U) }{2^R U}
+
\nu(\nu + \log U)\log(\nu + \log U).\label{psi22}
\end{align}

We now bound $\Psi_3$. For primes $p \equiv 1  \bmod{r}$, we have that $p$ splits completely in $\mathcal{O}_{\Q(\zeta_r)} $ as, say, $p \mathcal{O}_k = \pi_{p,1} \cdot \pi_{p,2} \cdot  ... \cdot \pi_{p, \varphi(r)}$.
Let $$
b_q := \frac{\mu(q)}{\varphi(q)} \prod_{p|q} (n_{k,p}-1), 
$$
and 
consider the associated Dirichlet series
\begin{align*}
f(s,k)
& := \sum_q \frac{\mu(q)}{\varphi(q)q^s} \prod_{p|q} (n_{k,p}-1) \\
& = \prod_{p \equiv 1  \bmod{r}} \left (1 -\frac{n_{k,p}-1 }{(p-1)p^s} \right )  \\
& =  \prod_{\substack{p \equiv 1   \bmod{r} \\ p= \pi_{p,1} \cdot ... \cdot \pi_{p,\varphi(r)} }} \left  (1 - \frac{ \left (\frac{k}{\pi_{p,1}} \right )_r+ ... + \left (\frac{k}{\pi_{p,\varphi(r)}} \right )_r}{(p-1)p^s} \right ).
\end{align*}
(The last equality above uses \eqref{blahblah} and Proposition \ref{correspondence}.)
If $s=0$, then $f(s,k)= \mathfrak{S}_{0,1}(k)$ which is absolutely convergent so $f(s,k)$ has no poles with $\Re(s)>0$. 

We relate this to the Hecke $L$-function associated to $\legendre{k}{\cdot}_r$
\begin{align*}
L\left (s+1, \left (\frac{k}{\cdot} \right )_r \right ) & = \prod_{\pi \textrm{ prime ideal in }\Z[\zeta_r]} \left (1 - \frac{ \left (\frac{k}{\pi}   \right )_r}{N(\pi)^{s+1} } \right )^{-1}.
\end{align*}

Indeed
$$
h(s,k) := L\left (s+1, \left (\frac{k}{\cdot} \right )_r \right ) f(s,k)
=
\prod_p \left( 1 + O \left(\frac{1}{p^{2(s+1)}}\right)\right).
$$
thus $h$ is absolutely bounded for all $\Re(s)>-1/2$. Applying Perron's formula (Lemma~\ref{lem: perron}), we have
\begin{align*}
\sum_{y_1 \le q \le y_2} b_q & = \frac{1}{2\pi i} \int_{C-iT}^{C+iT} L^{-1}\left (s+1, \left (\frac{k}{\cdot} \right )_r \right )h(s,k) \frac{y_2^s - y_1^s}{s} \ \mathrm{d} s \\
& \quad  +O \left ( \sum_{j=1}^2 \sum_q |b_q| \left (\frac{y_j}{q} \right )^C \min \left  \{1 , T^{-1} \left |\log \left (\frac{y_j}{q}  \right ) \right |^{-1} \right \} \right ).
\end{align*}
for any $C,T>0$. By Lemma~\ref{lem: zero free region} we may move the line of integration back to $[\sigma-iT,\sigma+iT]$ for  $$
\sigma > 1-\frac{c'}{\phi(r) \log \left(\vert\Delta_{\Q(\zeta_r)}\vert (T +3)\right)}.
$$ without encountering a pole of the integrand.
Bounding the contributions from the horizontal and vertical line segments of the contour in the standard way (mimicking the proof of the prime number theorem), we achieve the bound
\begin{equation} \label{eq: Psi_3 bound}
\Psi_{3} \ll y \exp \left (-c_3 \sqrt{2^{\nu} U} \right ),
\end{equation}
for some fixed constant $c_3>0$.

So far we have thus computed

\[
\begin{array}{ll}
\Psi_1 & \ll \frac{y}{Q_1^{1-\varepsilon}}  + U^{1+\varepsilon}\\
\Psi_2 & \ll \nu \left(  y^{1/2 + \epsilon} + \frac{ y^{2 \phi(r)}(\nu + \log U)\log(\nu + \log U) }{y^{2\phi(r) +\epsilon}} + \nu(\nu + \log U)\log(\nu + \log U) \right)\\
\Psi_{3} & \ll y \exp \left (-c_3 \sqrt{2^{\nu} U} \right ),\\
\end{array}
\]
see~\eqref{eq: Psi_1 bound},~\eqref{psi21},~\eqref{psi22} and~\eqref{eq: Psi_3 bound}.
Recalling that we need $U > Q_1 = (\log x)^{c_1}$, $2^\nu U > U$, and $\nu > \lfloor \log_{2} ( y^{2\phi(r) +\epsilon}/U) \rfloor$, we can choose $\epsilon >0$ small enough and choose $U$ and $\nu$ to be small enough positive real powers of $y$ 
to deduce that
\begin{equation} \label{eq: Psi square average bound}
\sum_{\substack{k \le y }} |\Psi(k)|^2 \ll \frac{y}{(\log x)^{c_4} },
\end{equation}
for any given $c_4 >0$.

\section{Error terms from the major arcs} \label{section: error in major arc}
The final major arc contribution which we need to estimate is
\[
\mathcal{E} :=
\sum_{k \leq y} \left \vert 
\int_{\mathfrak{M}} E_1(\alpha)S_2(\alpha) e(-\alpha k)\mathrm{d}\alpha \right \vert^2,
\]
where
\begin{align*}
E_1(\alpha) &= \frac{1}{\varphi(q)} \sum_{\substack{\chi \bmod q   }} \tau(\overline{\chi}) \chi(a) \sum_{\substack{m \le z\\  (m,q)=1}}^{\#} \chi(m) \Lambda(m) e(\beta m),\\
S_2(\alpha) &= \sum_{\substack{b \bmod{[M_0,q]}\\ b \equiv n_0 \bmod{M_0}}} e \left(-\frac{a}{q}b^r\right) 
\sum_{\substack{ n \leq x \\ n \equiv b \bmod{[M_0,q]}}} e(\beta n^r).
\end{align*}

Now we apply Lemma~\ref{lem: bessel} with 
\[
\phi_k = e(- \alpha k) \quad{}\text{ and } \quad{} \xi(\alpha) = \begin{cases} S_2(\alpha)E_1(\alpha) \text{ if } \alpha \in \mathfrak{M}\\ 0 \text{ otherwise},\end{cases} 
\]
to get
\begin{align*}
\mathcal{E}  & \ll \int_{\mathfrak{M}} |S_2(\alpha)   E_1(\alpha)|^2 \ \mathrm{d} \alpha \\
&  \ll \sup_{\alpha \in \mathfrak{M}} |S_2(\alpha)|^2 \int_{\mathfrak{M}} |E_1(\alpha)|^2 \ \mathrm{d} \alpha \\
& \ll x^2 \int_{\mathfrak{M}} |E_1(\alpha)|^2 \ \mathrm{d} \alpha.
\end{align*}

Breaking the $m$-sum of $E_1(\alpha)$ into dyadic intervals, of which there are $O(\log z)$,  and applying the Cauchy--Schwarz inequality, we have
\begin{align*}
\int_{\mathfrak{M}} |E_1(\alpha)|^2 \ \mathrm{d} \alpha
&\ll
\log z  \sum_{q \leq Q_1} \sum_{\substack{a=1\\\gcd(a,q)=1}}^q 
\int_{\vert \beta \vert < \frac{1}{qQ_2}} \left \vert
\frac{1}{\phi(q)} \sum_{\chi \bmod{q}} \tau(\overline{\chi}) \chi(a) \sum^{\#}_{\substack{m \sim M \\ \gcd(m,q) =1}}
\chi(m) \Lambda(m) e(\beta m) \right \vert^2 \mathrm{d}\beta\\
&\ll \log z  \sum_{q \leq Q_1} \frac{q}{\phi(q_1)} \sum_{\chi \bmod{q}} \int_{\vert \beta \vert < \frac{1}{qQ_2}} \left\vert\sum^{\#}_{\substack{m \sim M \\ \gcd(m,q) =1}}
\chi(m) \Lambda(m) e(\beta m) \right \vert^2 \mathrm{d}\beta,
\end{align*}
for some $M \in [1, z]$. 

When $M^{1- \epsilon}> (\log x)^{c_1}Q_2$, we apply Lemma~\ref{lem: gallagher} to get
\begin{align*}
\int_{\vert \beta \vert < \frac{1}{qQ_2}} \left\vert\sum^{\#}_{\substack{m \sim M \\ \gcd(m,q) =1}}
\chi(m) \Lambda(m) e(\beta m) \right \vert^2 \mathrm{d}\beta
&\ll
\frac{1}{(qQ_2)^2} \int_{M - \frac{qQ_2}{2}}^{2M} \left \vert
\sum^{\#}_{\max\{t, M\} < m \leq \min\{t + \frac{qQ_2}{2}, 2M\}} \chi(m) \Lambda(m) \right \vert^2 \mathrm{d}t.
\end{align*}
We may replace the right hand side by $\mathfrak{J}(q, \frac{Q_2}{2})$ (as defined in Lemma~\ref{lem: mikawa}) upto an error of size
\[
\ll \frac{1}{(qQ_2)^2}\int_{M - \frac{qQ_2}{2}}^{M} \left \vert \sum^{\#}_{M \leq m \leq t + \frac{qQ_2}{2}} \chi(m) \Lambda(m) \right \vert^2
\ll
qQ_2 (\log M)^2.
\]
Applying Lemma \ref{lem: mikawa} with $\Delta = \frac{Q_2}{2}$ and $N = M$ (which is valid because we know that $\Delta \asymp x^{r-1-\varepsilon} > x^{r/5+\epsilon}> M^{1/5 + \epsilon}$), we get
\[
\frac{1}{(qQ_2)^2} \int_{M}^{2M} \left \vert
\sum^{\#}_{\max\{t, M\} < m \leq \min\{t + \frac{qQ_2}{2}, 2M\}} \chi(m) \Lambda(m) \right \vert^2 \mathrm{d}t
\ll M (\log M)^{-A}.
\]
If $M^{1- \epsilon}< (\log x)^{c_1} Q_2$ then the assumptions of Lemmas \ref{lem: mikawa} and \ref{lem: gallagher} are not satisfied with the choice of $\Delta$ and $N$ as above. In this case, we extend the $\beta$ integral to the range $\vert \beta \vert < \frac{1}{qM^{\frac{1}{2}}}$ and choose $\Delta = M^{\frac{1}{2}}$ and $N=M$ (unless $M$ is smaller than any power of $x$ in which case the following bound is trivial).
Therefore we conclude
\begin{align} 
\int_{\mathfrak{M}} |E_1(\alpha)|^2 \ \mathrm{d} \alpha & \ll \sum_{q < Q_1} \frac{q}{\varphi(q)} (qQ_2)^{-2} \mathfrak{J} (q,Q_2/2) + Q_1^3 Q_2(\log x)^2  \nonumber \\
& \ll \sum_{q < Q_1} \frac{q}{\varphi(q)} z (\log z)^{-A},\nonumber
\end{align}
for any $A>0$. Thus
\begin{equation}\label{eq: bound for e}
\mathcal{E} \ll \frac{x^2 z}{(\log x)^{c_8} }.
\end{equation}

\section{The  minor arcs} \label{section: minor arc}
Finally we estimate the minor arc contribution
\begin{align*}
 \sum_{\substack{k \le y }}   \left | \int_{\mathfrak{m} }  S_1(\alpha) S_2(\alpha) e(-\alpha k) \ \mathrm{d} \alpha \right |^2  = \sum_{\substack{k \le y  }} \left | \int_{\mathfrak{m} }  \sum_{m \le z} \Lambda(m)e(\alpha m) \sum_{\substack{n \le x\\ n \equiv n_0 \bmod{M_0}}} e(-\alpha(n^r +k)) \ \mathrm{d} \alpha \right |^2.
\end{align*}
Using Lemma~\ref{lem: bessel} as in the previous section, we bound the minor arc contribution by
$$
 \sup_{\alpha \in \mathfrak{m}} |S_2(\alpha)|^2 \int_0^1 |S_1(\alpha)|^2  \ \mathrm{d} \alpha \ll ( z \log z ) \sup_{\alpha \in \mathfrak{m}} |S_2(\alpha)|^2 .
$$
Weyl differencing (Lemma~\ref{lem: weyl}) gives us
$$
|S_{2}(\alpha)|^{2^{r-1}} \ll x^{2^{r-1}-r}  \sum_{-x < \ell_1, ... ,\ell_{r-1} <x} \min \left \{ x, \frac{1}{\lVert r! \alpha \ell_1 ... \ell_{r-1}  \rVert} \right \}.
$$
As consequence of Dirichlet's Diophantine approximation theorem, we have the inequality
$$
\left  | \alpha - \frac{a}{q}  \right | \le\frac{1}{2 r! x^{r-1} q},
$$
where $(a,q)=1$ and $1 \le q \le 2 r! x^{r-1}$. Since $\frac{1}{2r!x^{r-1}q} < \frac{1}{qQ_2}$, for $\alpha \in \mathfrak{m}$, we must have that $q > Q_1$. For $-x < \ell_1, ... , \ell_{r-1} <x$, we have
$$
\left |r!  \ell_1 ... \ell_{r-1} \left (\alpha -  \frac{a}{q} \right ) \right| \le \frac{1}{2q},
$$
which implies
$$
\frac{1}{\lVert r!  \ell_1 ... \ell_{r-1} \alpha \rVert } \le \frac{2}{\lVert r! \ell_1 ... \ell_{r-1} a/q \rVert}.
$$
Therefore
\begin{align*}
 \sum_{-x < \ell_1, ... ,\ell_{r-1} <x} &  \min \left \{ x, \frac{1}{\lVert r! \alpha \ell_1 ... \ell_{r-1}  \rVert} \right \}   \ll x \sum_{\substack{-x < \ell_1, ... ,\ell_{r-1} <x \\ q|r!\ell_1 ... \ell_{r-1}}} 1 + \sum_{\substack{-x < \ell_1, ... ,\ell_3 <x \\ q \nmid r!\ell_1 ... \ell_{r-1}}}  \frac{2}{\lVert r!  \ell_1 ... \ell_{r-1} a/q \rVert}.
\end{align*}
Letting $\widetilde{q} = \frac{q}{\gcd(q,r!)}$ and $\tau_d(n)$ for the number of ways of writing $n$ as a product of $d$ natural numbers, we have
\begin{align*}
\sum_{\substack{-x < \ell_1, ... ,\ell_{r-1} <x \\ q|r!\ell_1 ... \ell_{r-1}}} 1  
&\ll 
\sum_{\substack{n \leq x^{r-1}\\ \widetilde{q} \mid n}} \tau_{r-1}(n)
\ll
(r-1)^{r-2} \frac{x^{r-1}}{\widetilde{q}} \log^{r-2}(x).
\end{align*}
Since $q > Q_1$, this is at most $\frac{x^{r-1}}{(\log x)^{c_{10}}}$.
Moreover, we have
$$
\sum_{\substack{-x < \ell_1, ...  ,\ell_{r-1} <x \\ q \nmid r!\ell_1 ... \ell_{r-1}}}  \frac{2}{\lVert r!  \ell_1 ... \ell_{r-1} a/q \rVert} \ll 
x^{r-1 + \epsilon}
$$
Hence, we get
$$
\sup_{\alpha \in \mathfrak{m}} |S_2(\alpha)|^{2^{r-1}} \ll   x^{2^{r-1}-r} ( x^{r})(\log x)^{-c_{10}}
$$
and so 
\[ \sup_{\alpha \in \mathfrak{m}} |S_2(\alpha)|^{2} \ll  x^{2}(\log x)^{-c_{11}}.\]

Recalling~\eqref{eq: z}, we thus get
\begin{equation} \label{eq: minor arc bound}
\sum_{\substack{k \le y }}   \left | \int_{\mathfrak{m} }  S_1(\alpha) S_2(\alpha) e(-\alpha k) \ \mathrm{d} \alpha \right |^2 \ll (x^r \log x ) \sup_{\alpha \in \mathfrak{m}} |S_2(\alpha)|^2  \ll x^{r+2}(\log x)^{-c_{11}}.
\end{equation}
\section{Putting it all together}\label{sec:together}
We are finally ready to estimate 
\[
E := \sum_{k \leq y}
\left \vert \int_0^1 S_1(\alpha) S_2(\alpha) e(-\alpha k)  - \mathfrak{S}_{n_0, M_0}(k) x \right \vert^2.
\]
By the circle method and \eqref{eq: marjor arc}, we have that 
\[
E = \sum_{k \leq y} \left \vert  x |\Psi(k)| +(\log x)^{c_2} x^{1 - \epsilon'} +  \int_{\mathfrak{M}} E_1(\alpha)S_2(\alpha) e(-\alpha k)\mathrm{d}\alpha +  \int_{\mathfrak{m} }  S_1(\alpha) S_2(\alpha) e(-\alpha k) \ \mathrm{d} \alpha \right \vert^2.
\]
Applying Cauchy--Schwarz and the bounds \eqref{eq: Psi square average bound}, the trivial bound,  \eqref{eq: bound for e} and \eqref{eq: minor arc bound},  we have
\begin{align*}
E 
&\ll \frac{x^2y}{(\log x)^{c_4}}+ yx^{2-\epsilon'''} + \frac{x^2 z}{(\log x)^{c_8}} + x^{r+2}(\log x)^{-c_{11}}\\
& \ll_B \frac{y x^2}{(\log x)^B},
\end{align*}
which completes the proof.

\section{Integral Hasse principle on average for certain generalised Ch\^atelet surfaces}\label{sec:chats}
 Let $a \in \Z - \{0,1\}$ be squarefree, let  $r \geq 3$ be an integer,  and let  $k \in \Z$. 
Let $ \mathcal{X}_{a,r,k} \subset \left( \mathcal{O}_{\Q(\sqrt{a})} \otimes \A^1_{\Z} \right) \times  \A^1_{\Z,t}$ be given by the equation
\[N_{\Q(\sqrt{a})/\Q}(\textbf{z}) = t^r + k \neq 0.\]
\begin{theorem} \label{chat} Let $a \in \Z - \{0,1\}$ be squarefree and such that  2 does not ramify in $\mathcal{O}_{\Q(\sqrt{a})}$.  Let $r \in \NN$ be a prime such that $p \not\equiv 1 \bmod{r}$ for all primes $p|a$. Then for 100\%  of $k \in \NN$ (ordered naively by size) we have $\mathcal{X}_{a,r,k}(\Q) \neq \emptyset$. If, moreover, $\mathcal{O}_{\Q(\sqrt{a})}$ has narrow class number at most 2, then  for 100\%  of $k \in \NN$ (ordered naively by size) we have $\mathcal{X}_{a,r,k}(\Z) \neq \emptyset$.
\end{theorem}
\begin{proof} Consider the open affine variety
\[ \begin{array}{rrll}
\mathcal{X}_{a,r}: &  N_{\Q(\sqrt{a})/\Q}(\textbf{z})  = &  t^r + x \neq 0 & \subset \left( \mathcal{O}_{\Q(\sqrt{a})} \otimes \A^1_\Z \right) \times \A_{\Z, t}^1 \times \A^1_{\Z, x},\\
\end{array}\]
There are natural maps projecting to the $t$- and the $x$-coordinates given by
\[ \begin{array}{rl}
\textup{pr}_t:& \mathcal{X}_{a,r}   \to \A_{\Z, t}^1,\\
\textup{pr}_x:& \mathcal{X}_{a,r} \to \A_{\Z, x}^1.
\end{array}\]
We write $\mathcal{X}_{a,r,x_0} := (\textup{pr}_x)^{-1}(\{x_0\})$ for the fibre of $\textup{pr}_x$ above $x_0 \in \A_{\Z, x}^1$ and $\mathcal{X}_{a,r,x_0,t_0}  := \xi^{-1}((x_0,t_0))$ for the fibre of $\xi$ above $(x_0,t_0) \in \mathcal{U}_{a,r}$, where $\mathcal{U}_{a,r} : t^r + x \neq 0 \subset \A_{\Z, t}^1 \times  \A_{\Z, x}^1$ and $\xi: \mathcal{X}_{a,r} \to \mathcal{U}_{a,r}$ is the natural map.

Let $S := \{p \textup{ prime: $p$ is ramified in $\mathcal{O}_{\Q(\sqrt{a})}$}\}$.
For $p \in S$, we let $U_p \subset \mathcal{X}_{a,r}(\Z_p)$ be the set of ``primitive" solutions in the following sense: a point $(\textbf{z}_p, t_p, x_p) \in (\mathcal{O}_K \otimes \Z_p, \Z_p, \Z_p)$ is in $U_p$ if and only if $ N_{\Q(\sqrt{a})/\Q}(\textbf{z}_p)  =t_p^r + x_p \in \Z_p^\times$.  

 Consider the set 
\[ N_{S}:= 	\left\{ k  \in \NN:  \prod_{p \in S} (U_p \cap \mathcal{X}_{a,r,k}(\Z_p)) \neq \emptyset \right  \}.\]
We claim that $N_S = \NN$. (We remark that we use the restrictions on $a$ and $r$ in the statement of the theorem in order to prove this claim. Relaxing or removing these conditions would still yield a positive density result at the end, but not necessarily with a density 1.) Indeed, since for all  $p \in S$ we have by assumption that  $p \not\equiv 1 \bmod{r}$, and since this is equivalent to $\gcd(r, p-1)=1$, it follows that the map $u \mapsto u^r$ is a bijection on $\Z/p\Z$. Hence, for each $k \in \NN$,  we can find some $\overline{t} \in (\Z/p\Z)^\times$ and a square $\square \in \Z$ not divisible by $p$ satisfying the congruence equation $\overline{t}^r + k - \square \equiv 0 \bmod{p}$: if $k \not \equiv 1 \bmod{p}$, then we can take $\square = 1$, while if  $k \equiv 1 \bmod{p}$ and $p \geq 5$, then we can take $\square = 4$; in either case, by Hensel's lemma  we can then lift this solution $\overline{t}$ to some $t_p \in \Z_p$ with $t_p^r + k - \square = 0$; it follows that, with respect to the standard basis $\{1, \sqrt{a}\}$ for $\Q(\sqrt{a})$, the $\Z_p$-point $((\sqrt{\square} + \sqrt{a} \cdot 0), t_p) \in (\Z_p + \sqrt{a}\Z_p) \times \Z_p \subset (\mathcal{O}_{\Q(\sqrt{a})} \otimes \Z_p) \times \Z_p$ lies in $U_p \cap \mathcal{X}_{a,r,k}(\Z_p)$. If $p=3$ and $k \equiv 1 \bmod{3}$, we can instead note that the solution $\overline{y} \equiv 1$ to the congruence equation $\overline{y}^2 - k \equiv 0 \bmod{3}$ can be lifted to a $\Z_3$-solution $y_3$ to $y^2-k = 0$. In this case, the $\Z_3$-point $((y_3 + \sqrt{a} \cdot 0), 0) \in (\Z_3 + \sqrt{a}\Z_3) \times \Z_3 \subset (\mathcal{O}_{\Q(\sqrt{a})} \otimes \Z_3) \times \Z_3$ lies in $U_3 \cap \mathcal{X}_{a,r,k}(\Z_3)$. 
This proves  our claim that the set $N_S$ is exactly $\NN$.

Now consider the fibre product 
\[
\begin{tikzcd}
(\mathcal{O}_{\Q(\sqrt{a})} \otimes \Z_p) \times_{\Z_p} (\Z_p \times \Z_p) \arrow[r, "\phi_1"] \arrow[d, "\phi_2"] & \Z_p \times \Z_p \arrow[d, "t^r+x" ] \\
\mathcal{O}_{\Q(\sqrt{a})} \otimes \Z_p \arrow[r, "N_{\Q(\sqrt{a})/\Q}(\textbf{z})" ] & \Z_p.
\end{tikzcd}
\]
It is easy to see that the inverse image of $\Z_p^\times$ under the composite continuous map $N_{\Q(\sqrt{a})/\Q}(\textbf{z}) \circ \phi_2 = (t^r+x) \circ \phi_1$ is precisely $U_p$.
Since $p \Z_p$ is open in $\Z_p$, it follows that $\Z_p^\times$ is closed in $\Z_p$. Hence, $U_p$ is compact, and thus so are $\prod_{p \in S} \textup{pr}_x(U_p)$ and  $\prod_{p \in S} \textup{pr}_t(U_p)$. By compactness, it follows that there exists some small enough $\nu > 0$ and some points $(w_p^{(i_p)})_{p \in S} \in  U_p$ for $i_p \in \{ 1, ..., n_p\}$ and for all $p \in S$ such that 
\begin{itemize}
\item we can cover $\prod_{p \in S} \textup{pr}_x(U_p)$ and  $\prod_{p \in S} \textup{pr}_t(U_p)$ by open balls of radius at most $\nu$ around some $(w_p^{(i_p)})_{p \in S}$,  and 
\item if $p \in S$ and $t_p, x_p \in \Z_p$ is such that $| x_p - \textup{pr}_x(w_p^{(i_p)})|_p, | t_p - \textup{pr}_t(w_p^{(i_p)})|_p < \nu$, then $\mathcal{X}_{a,r,x_p, t_p}(\Q_p) \neq \emptyset$. 
\end{itemize}

Let $C_S:= \{ \prod_{p \in S} w_p^{(i_p)} : i_p \in \{1, ..., n_p\}\}$ be the set of all possible combinations across $p \in S$ of the centres of the balls as above.
By the Chinese Remainder Theorem, for each $\prod_{p \in S} w_p^{(i_p)}  \in C_S$ there exist integers $u^{(\prod_{p \in S} i_p)}, v^{(\prod_{p \in S} i_p)} \in \Z$ such that $|u^{(\prod_{p \in S} i_p)} - \textup{pr}_x(w^{(i_p)}_p)|_p < \nu$ and  $|v^{(\prod_{p \in S} i_p)} - \textup{pr}_t(w^{(i_p)}_p)|_p < \nu$,  for all $p \in S$. Then any $k \in N_{S} = \NN$ must satisfy $| k - u^{(\prod_{p \in S} i_p)}|_p < \nu$, for all $p \in S$, for some $\prod_{p \in S} w_p^{(i_p)}  \in C_S$. Note that, by our definition of $N_S$ and by the compactness conditions, any such $\prod_{p \in S} w_p^{(i_p)}  \in C_S$ satisfies $((\textup{pr}_t(w_p^{(i_p)}))^r + \textup{pr}_x(w_p^{(i_p)})) \in \Z_p^\times $  for all $p \in S$, thus implying that $\gcd((v^{(\prod_{p \in S} i_p)})^r + u^{(\prod_{p \in S} i_p)}, \prod_{p\in S} p^{\lfloor-\log_p(\nu)\rfloor}) = 1$  for all $p \in S$. 

Hence, we can use Corollary \ref{cor} with $n_0 := v^{(\prod_{p \in S} i_p)}$ and $M_0 := \prod_{p\in S} p^{\lfloor-\log_p(\nu)\rfloor}$ to deduce that for 100\% of $k \in N_S = \NN$  with $|k -  u^{(\prod_{p \in S} i_p)}|_p < \nu $, there exists some integer $n \in \NN$ such that $| n - v^{(\prod_{p \in S} i_p)}|_p < \nu$ for all $p \in S$ and such that $n^r + k = q$ is prime. We note that, since  $((v^{(i_p)}_p)^r + u^{(i_p)}_p) \in \Z^\times_p$  for all $p \in S$, we have that $\gcd(q, a) = 1$.

It is now standard to show that $\mathcal{X}_{a,r,k,n} (\A_\Q) \neq \emptyset$: indeed, $\mathcal{X}_{a,r,k,n} (\RR) \neq \emptyset$ since $q > 0$; for $p \in S$, this follows by our compactness assumptions above; for $p \notin S$ and $p \neq q$, we have that $q \in \Z_p^\times$ and since $p$ is unramified in $\mathcal{O}_{\Q(\sqrt{a})}$, it is well-known that norms are surjective onto the units, and thus $q \in N_{\Q(\sqrt{a})/\Q}(\mathcal{O}_K \otimes \Z_p)$; finally, for $p = q$, this follows from the global reciprocity law for Hilbert symbols and by the fact that the relevant local Hilbert symbols at all the other places $\neq q$ are all $1$.

Hence, we have just shown that the conic
\[ \mathcal{X}_{a,r,k,n}: N_{\Q(\sqrt{a})/\Q}(\textbf{z}) = q\]
is everywhere locally soluble. By the Hasse principle for conics, this implies that $\mathcal{X}_{a,r,k,n}(\Q) \neq \emptyset$.

If we assume further that $\mathcal{O}_{\Q(\sqrt{a})}$ has narrow class number at most 2, then we can apply \cite[Proposition 1.2]{Mi} and \cite[Remark 3.1]{Mi} to conclude that $\mathcal{X}_{a,r,k,n}(\Z) \neq \emptyset$, and thus that $\mathcal{X}_{a,r,k}(\Z) \neq \emptyset$. (Indeed, note that $\left(\frac{a}{q}\right) = 1$ as $\gcd(q,a) = 1$ and as the valuation of the norm at $q$ is odd; this implies that $q$ is not inert in $\mathcal{O}_{\Q(\sqrt{a})}$ and so \cite[Proposition 1.2]{Mi} applies.)

Finally, noting that any $k \in N_S = \NN$ satisfies  $|k -  u^{(\prod_{p \in S} i_p)}|_p < \nu $ for some $\prod_{p \in S} w_p^{(i_p)}  \in C_S$, the above argument shows that for 100\% of $k \in \NN$ we have that $\mathcal{X}_{a,r,k}(\Q) \neq \emptyset$ and, if $\mathcal{O}_{\Q(\sqrt{a})}$ has narrow class number at most 2, that  $\mathcal{X}_{a,r,k}(\Z) \neq \emptyset$, as required.
\end{proof}

\begin{remark}
By imposing analogous restrictions on a given cyclic extension $K/\QQ$, a similar proof as the one above would show that a positive proportion of the varieties defined by
\[
N_{K/\QQ}(\z) = t^r + k \neq 0,
\] have a $\QQ$-point. Moreover, it is highly likely that by imposing further restrictions on $K$ one can also recover a 100\% type result for $\QQ$ points.
\end{remark}

\begin{remark} The proof of Theorem \ref{chat} also yields the following result:  let $a \in \Z - \{0,1\}$ be squarefree and let $r \in \NN$ be a prime. Let $S := \{p \textup{ prime: $p$ is ramified in $\mathcal{O}_{\Q(\sqrt{a})}$}\}$. For $p \in S$,  let $U_p \subset \mathcal{X}_{a,r}(\Z_p)$ be the set of ``primitive" solutions given by solutions $(\textbf{z}_p, t_p, x_p) \in (\mathcal{O}_K \otimes \Z_p, \Z_p, \Z_p)$ such that $ N_{\Q(\sqrt{a})/\Q}(\textbf{z}_p)  =t_p^r + x_p \in \Z_p^\times$.  Then for 100\%  of $k \in \NN$ (ordered naively by size) we have
\[ \prod_{p \in S} (U_p \cap \mathcal{X}_{a,r,k}(\Z_p)) \neq \emptyset  \Longrightarrow  \mathcal{X}_{a,r,k}(\Q) \neq \emptyset.\]
 If, moreover, $\mathcal{O}_{\Q(\sqrt{a})}$ has narrow class number at most 2, then  for 100\%  of $k \in \NN$ (ordered naively by size) we have 
\[ \prod_{p \in S} (U_p \cap \mathcal{X}_{a,r,k}(\Z_p)) \neq \emptyset  \Longrightarrow  \mathcal{X}_{a,r,k}(\Z) \neq \emptyset.\]
Our further restrictions on $a$ and $r$ in the statement of Theorem \ref{chat} are there in order to ensure that $ \prod_{p \in S} (U_p \cap \mathcal{X}_{a,r,k}(\Z_p)) \neq \emptyset$ for 100\% of the $k$'s. By ignoring these restrictions, it is always possible to get a lower bound on the density  of $k$'s satisfying  $ \prod_{p \in S} (U_p \cap \mathcal{X}_{a,r,k}(\Z_p)) \neq \emptyset$ by considering, for example, all $k$'s  congruent to 1 modulo a high enough power of $\prod_{p \in S} p$, because any such $k$ will be both a unit and a square in $\Z_p$ for all $p \in S$.
\end{remark}

\end{document}